\documentclass[12pt]{amsart}
\usepackage[utf8]{inputenc}
\usepackage{tikz}
\usepackage{amsthm}
\usepackage{amsmath}
\usepackage{amssymb}
\usepackage{amsthm}
\usepackage{amscd}
\usepackage{mathtools}
\usepackage[margin=1in]{geometry}
\usepackage{comment}
\usepackage{tikz-cd}
\usepackage{booktabs}
\usepackage{array}
\usepackage{blkarray}
\newcolumntype{L}[1]{>{\raggedright\arraybackslash}p{#1}}

\usepackage{biblatex}
\usepackage{hyperref}

\newtheorem{thm}{Theorem}[section]
\newtheorem{innercustomthm}{Theorem}
\newenvironment{customthm}[1]
{\renewcommand\theinnercustomthm{#1}\innercustomthm}
{\endinnercustomthm}
\newtheorem*{thm*}{Theorem}
\newtheorem{lem}[thm]{Lemma}
\newtheorem{conj}[thm]{Conjecture}
\newtheorem{prop}[thm]{Proposition}
\newtheorem{cor}[thm]{Corollary}

\theoremstyle{definition}

\newtheorem{remark}[thm]{Remark}
\newtheorem{assumption}[thm]{Assumption}
\newtheorem{example}[thm]{Example}

\numberwithin{equation}{section}

\newcommand\xto[1]{\xrightarrow{#1}}
\newcommand{\mf}{\mathfrak}
\newcommand{\mb}{\mathbb}
\newcommand{\mc}{\mathcal}
\newcommand{\mbf}{\mathbf}
\newcommand{\ssc}{\mathrm{ssc}}
\newcommand{\wt}{\widetilde}

\def\Hom{\mbox{\rm Hom}}

\newcommand{\colorZ}[1]{{{#1}}}
\newcommand{\colorX}[1]{{{#1}}}
\DeclareMathOperator\rank{rank}
\DeclareMathOperator\Gr{Gr}
\DeclareMathOperator\GL{GL}
\DeclareMathOperator\SL{SL}
\DeclareMathOperator\Sym{Sym}
\DeclareMathOperator\ad{ad}
\DeclareMathOperator\Spec{Spec}
\DeclareMathOperator\depth{depth}
\DeclareMathOperator\grade{grade}
\DeclareMathOperator\pdim{pdim}
\DeclareMathOperator\Tor{Tor}
\DeclareMathOperator\Aut{Aut}
\newcommand{\Id}{\mathrm{Id}}
\newcommand{\Plucker}{\mathrm{Pl\ddot{u}cker}}

\newcommand{\Rgen}{\widehat{R}_\mathrm{gen}}
\newcommand{\Fgen}{\mathbb{F}^\mathrm{gen}}
\newcommand{\margin}[1]{\scalebox{0.7}{#1}}
\newcommand{\cotimes}{\,\widehat{\otimes}\,}

\begin{document}
\title{An ADE correspondence for grade three perfect ideals}

\author{Lorenzo Guerrieri, Xianglong Ni, Jerzy Weyman}

\dedicatory{In memory of David Buchsbaum}

\maketitle

\begin{abstract}
	Using the theory of ``higher structure maps'' from generic rings for free resolutions of length three, we give a classification of grade 3 perfect ideals with small type and deviation in local rings of equicharacteristic zero, extending the Buchsbaum-Eisenbud structure theorem on Gorenstein ideals and realizing it as the type D case of an ADE correspondence. We also deduce restrictions on Betti tables in the graded setting for such ideals.
	
	
\end{abstract} 

\setcounter{tocdepth}{2}
\tableofcontents

\section{Introduction}\label{sec:intro}


We say an ideal $I$ in a local Noetherian $\mb{C}$-algebra $(R,\mf{m})$ is \emph{perfect} if $\grade I = \pdim_R R/I$, where $\grade I \coloneqq \depth(I,R)$. In the event that $R$ is a regular local ring, $I$ being perfect is equivalent to $R/I$ being a Cohen-Macaulay $R$-module by the Auslander-Buchsbaum formula. It is well-known that all perfect ideals of grade two are determinantal. More precisely one has the following corollary of Hilbert-Burch:

\begin{thm*}
	Let $S$ be the polynomial ring on the variables $x_{i,j}$ ($1 \leq i \leq n$, $1 \leq j \leq n-1$), localized at the ideal of variables. Consider the complex
	\[
	\mb{G} \colon 0 \to S^{n-1} \xto{d_2} S^n \xto{d_1} S
	\]
	where $d_2 = [x_{i,j}]$ is the generic matrix in the variables $x_{i,j}$ and $(d_1)_{1,k}$ is $(-1)^k$ times the $k$-th $(n-1)\times (n-1)$ minor of $d_2$. This complex is acyclic.
	
	Letting $J$ denote the image of $d_1$, if $I \subset R$ is a perfect ideal of grade two minimally generated by $n$ elements, then there exists a local homomorphism $\varphi \colon S \to R$ such that $\varphi(J)R = I$, or equivalently that $\mb{G} \otimes R$ resolves $R/I$.
\end{thm*}

For perfect ideals of grade three, theorems of this type are only known in a few special cases. The Betti numbers $(1,b_1,b_2,b_3)$ of such an ideal are determined by the type $r(R/I) \geq 1$ and the deviation $d(I) \geq 0$: the former is the minimal number of generators for the canonical module of $R/I$, thus equal to $b_3$, and the latter is by definition $b_1 - 3$. Then one has $b_2 = b_1 + b_3 - 1 = r(R/I) + d(I) + 2$. 

The case $r(R/I) = 1$ gives Gorenstein ideals, and these are characterized by the well-known structure theorem of Buchsbaum and Eisenbud: $I$ is generated by the $(n-1)\times(n-1)$ Pfaffians of an $n\times n$ skew matrix which appears as the differential $d_2$ in a resolution of $R/I$. From this, an analogous result for almost complete intersections (i.e. $d(I) = 1$) can be deduced from linkage, which was also carried out in \cite{Buchsbaum-Eisenbud77}.

The natural question to pose is whether all perfect ideals of grade 3 with a fixed type and deviation can be realized as specializations of some local generic example, as they do in these two cases. The smallest new case to consider would be when $r(R/I) = d(I) = 2$, so $R/I$ has Betti numbers $(1,5,6,2)$. Here one encounters an obstacle: it was observed in \cite{Brown87} that the Tor algebra multiplication
\[
m_{1,1} \colon \Tor_1(R/I,k) \otimes \Tor_1(R/I,k) \to \Tor_2(R/I,k)
\]
may or may not be zero. If $\varphi\colon R\to R'$ is a local homomorphism of local rings such that $I' = \varphi(I)R'$ is a perfect ideal of grade three in $R'$, then $\varphi$ induces an inclusion of residue fields $k \to k' \coloneqq R'/\mf{m'}$ through which $\Tor_*(R'/I',k') = \Tor_*(R/I,k) \otimes k'$. Consequently the property of $m_{1,1}$ being (non)zero is preserved under local specialization, meaning that there cannot be a single local generic example for perfect ideals with Betti numbers $(1,5,6,2)$.

Thus, in our attempt to generalize Hilbert-Burch and Buchsbaum-Eisenbud, we must make some sort of concession. One approach is to drop the condition that the generic example be local. This is the perspective taken by the ``genericity conjecture'' as discussed in \cite{WeymanICERM}. For the Betti numbers $(1,5,6,2)$, the conjecturally generic example $J(t)$ was explicitly described in \cite{CLKW20}. For instance, it was shown in \cite{Kustin22} that the example from \cite{Brown87} could be recovered as a specialization of $J(t)$.

Another approach is to keep the local condition, but to accommodate multiple generic examples for each sequence of Betti numbers. From the preceding discussion regarding $m_{1,1}$, we see that at least two generic examples are required for the Betti numbers $(1,5,6,2)$. We will see in Example~\ref{ex:E6} that this suffices. The generic example for the case $m_{1,1} \neq 0$ was already given in \cite{Brown87}, and it turns out that localizing $J(t)$ at the ideal of variables yields the generic example for the case $m_{1,1} = 0$. This is a consequence of our main classification result Theorem~\ref{thm:generic-examples}, which explains these two families in terms of the representation theory of $E_6$.

While we will address the former perspective by proving the genericity conjecture in Theorem~\ref{thm:non-local-gen}, the latter perspective proves to be much more revealing for the structure theory of perfect ideals. We will see that for certain Betti numbers, perfect ideals are naturally partitioned into disjoint families indexed by combinatorial data, and we enumerate them in Theorem~\ref{thm:counts}. In Examples~\ref{ex:D_n-1} and \ref{ex:D_n-2} we will recover that there is a single family for the Betti numbers $(1,n,n,1)$ ($n \geq 3$ odd) and $(1,4,n,n-3)$ ($n\geq 5$), but this behavior is the exception rather than the norm.

All of these results stem from a deep connection between free resolutions and representation theory that was originally uncovered in \cite{Weyman89} and \cite{Weyman18}. We motivate this connection by revisiting the well-understood Gorenstein case. Let $n = d(I)+3$ so that the Betti numbers are $(1,n,n,1)$. If $\mb{F}$ is an arbitrary minimal free resolution of $R/I$, then of course the differential $d_2\colon F_2 \to F_1$ need not be a skew matrix. As one sees from \cite{Buchsbaum-Eisenbud77}, the skew matrix appears after one identifies $F_2 \cong F_1^*$ using a choice of multiplication $F_1 \otimes F_2 \to F_3$ on the free resolution.

We can rephrase this as follows. The differential $d_2\colon F_2 \to F_1$ and the multiplication $F_1 \otimes F_2 \to F_3$ can be put side by side as an $n \times 2n$ matrix
\[
\begin{blockarray}{cccc}
	& \margin{$F_1^*$} & \margin{$F_1$} \\
	\begin{block}{c[ccc]}
		\margin{$F_2^*$} & d_2^* & w^{(2)}_1 \\
	\end{block}
\end{blockarray}
\]
where $w^{(2)}_1$ is the multiplication viewed as an isomorphism $F_1 \cong F_3^* \otimes F_1 \to F_2$. This matrix determines a map from $\Spec R$ to the orthogonal Grassmannian $OG(n,2n)$ of isotropic $n$-planes inside of $F_1^* \oplus F_1$ with the quadratic form $Q$ given by the evident pairing. To see this, there is an affine patch $N \subset OG(n,2n)$ consisting of those isotropic $n$-planes represented by an $n\times 2n$ matrix where the last $n\times n$ minor is non-vanishing. On such a matrix, after performing row operations so that the right $n \times n$ block is the identity matrix, the condition that $Q = 0$ is equivalent to the left block being skew. Moreover, on the affine patch $N$, the $(n-1)\times (n-1)$ pfaffians of the left skew matrix give Pl\"ucker coordinates cutting out a particular Schubert variety $X \subset OG(n,2n)$ of codimension 3. Thus the Buchsbaum-Eisenbud structure theorem yields a map $\Spec R \to N$, through which the local defining equations of $X$ at the ``origin'' in $N$ (representing the isotropic $n$-plane $F_1 \subset F_1^* \oplus F_1$) pull back to the generators of the ideal $I \subset R$. For a thorough explanation of this perspective, we refer to \cite{CVWthreetakes}. 

We will show that this formulation of the Buchsbaum-Eisenbud structure theorem realizes it as the type $D$ case of an ADE correspondence. In general, to deviation $d(I)$ and type $r(R/I)$, we associate the $T$-shaped graph with arms of length 1, $d(I)$, and $r(R/I)$. This graph $T$ is the Dynkin diagram
\begin{itemize}
	\item $A_n$ if $d(I) = 0$ and $r(R/I) = n-2$,
	\item $D_n$ if $d(I) = 1$ and $r(R/I) = n-3$ or vice versa,
	\item $E_6$ if $d(I) = r(R/I) = 2$,
	\item $E_7$ if $d(I) = 2$ and $r(R/I) = 3$ or vice versa,
	\item $E_8$ if $d(I) = 2$ and $r(R/I) = 4$ or vice versa.
\end{itemize}
Let $G$ be the associated simply-connected Lie group, and let $x_1 \in T$ be the node on the arm of length 1. We will show that generic examples of perfect ideals with the given type and deviation come from a certain codimension 3 Schubert variety inside of the homogeneous space $G/P_{x_1}$ where $P_{x_1}$ is the maximal parabolic for the node $x_1$.

The type $A$ case of the correspondence is uninteresting, as it occurs only when $d(I)=0$. This necessarily means that $I$ is generated by a regular sequence---in particular, $r(R/I) = 1$. The homogeneous space is $G/P_{x_1} = SL_4 / P_1 = \mb{P}^3$ in this case, and the Schubert variety $X$ is a point, which is indeed a complete intersection.

We mentioned above that there are two different families of perfect ideals with $d(I) = r(R/I) = 2$. This happens because the map $\Spec k \to \Spec R \to G/P_{x_1} = E_6/P_2$ lands in one of two strata of $X \subset E_6/P_2$, depending on whether the aforementioned multiplication on the Tor algebra is zero. We will describe these strata in \S\ref{sec:classify} and revisit this in Example~\ref{ex:E6}.

To achieve our goals, the main tool we will need is an appropriate generalization of the $n \times 2n$ block matrix leveraged in the Gorenstein case, and this is provided by the theory of ``higher structure maps'' originating from the study of generic free resolutions of length three. We now briefly explain this; more details will be given in \S\ref{bg:Rgen}. Given a complex
\[
\mb{F} \colon 0 \to F_m \xto{d_m} \cdots \to F_1 \xto{d_1} F_0,
\]
where $F_i = R^{f_i}$, we refer to the sequence $(f_0,f_1,\ldots,f_m)$ as the \emph{format} of $\mb{F}$. In the event that $\mb{F}$ is a minimal resolution over a local ring $R$, the $f_i$ are the ordinary Betti numbers of the module $H_0(\mb{F})$, but we will benefit from working in greater generality. For each fixed format $(f_0,f_1,f_2,f_3)$ of length three, a resolution $\mathbb{F}^\mathrm{gen}$ over a ring $\Rgen$ was constructed in \cite{Weyman89} and \cite{Weyman18} with the property that $\Fgen$ specializes to any free resolution of the given format.

The ring $\Rgen$ is a finitely generated $\mb{C}$-algebra if and only if the format $(f_0,f_1,f_2,f_3)$ is one listed in Table~\ref{table:Dynkin-types}. These are called \emph{Dynkin formats}, because the structure of $\Rgen$ is closely tied to a Kac-Moody Lie algebra $\mf{g}$ that is finite-dimensional (i.e. of Dynkin type) exactly in these cases. Henceforth we will always assume this to be the case.

\begin{table}[htbp]
	\caption{Length three formats with Noetherian $\Rgen$}
	\centering
	\begin{tabular}{*{4}cl}
		\toprule
		Type $D_n$ & Type $E_6$ & Type $E_7$ & Type $E_8$\\
		\cmidrule(lr){1-4}
		$(1,n,n,1)$ & $(1,5,6,2)$ & $(1,6,7,2)$ & $(1,7,8,2)$ & Format I (dual to VI) \\
		$(1,4,n,n-3)$ & & $(1,5,7,3)$ & $(1,5,8,4)$ & Format II (linked to I)\\
		$(n-3,n,4,1)$ & $(2,6,5,1)$ & $(3,7,5,1)$ & $(4,8,5,1)$ & Format III (dual to II)\\
		& $(2,5,5,2)$ & $(3,6,5,2)$ & $(4,7,5,2)$ & Format IV (linked to III)\\
		&& $(2,5,6,3)$ & $(2,5,7,4)$ & Format V (dual to IV)\\
		&& $(2,7,6,1)$ & $(2,8,7,1)$ & Format VI (linked to V)\\
		\bottomrule
	\end{tabular}
	\label{table:Dynkin-types}
\end{table}

We will explain the Lie algebra $\mf{g}$ and its relation to $\Rgen$ more precisely in \S\ref{bg:Rgen}. For now we comment that there are three representations inside of $\Rgen$ of particular interest: namely those generated by the entries of the differentials $d_i$ of $\mb{F}^\mathrm{gen}$. We call these the \emph{critical representations}. They have a graded decomposition in which each graded component is a representation of $\prod GL(F_i)$ where $F_i = \mb{C}^{f_i}$. We have displayed pieces of them below.
\begin{align*}
	W(d_3) &= F_2^* \otimes [F_3 \oplus \bigwedge^{f_0+1} F_1 \oplus \cdots]\\
	W(d_2) &= F_2 \otimes [F_1^* \oplus F_3^* \otimes \bigwedge^{f_0} F_1 \oplus \cdots]\\
	W(d_1) &= F_0^* \otimes [F_1 \oplus F_3^* \otimes \bigwedge^{f_0+2} F_1 \oplus \cdots]
\end{align*}
Given a homomorphism $w\colon \Rgen \to R$ specializing $\mb{F}^\mathrm{gen}$ to $\mb{F}$, restriction to $W(d_i)$ yields maps
\begin{align*}
	w^{(3)}\colon R \otimes [F_3 \oplus \bigwedge^{f_0+1} F_1 \oplus \cdots] &\to R \otimes F_2\\
	w^{(2)} \colon R \otimes [F_1^* \oplus F_3^* \otimes \bigwedge^{f_0} F_1 \oplus \cdots] &\to R \otimes F_2^*\\
	w^{(1)} \colon R \otimes [F_1 \oplus F_3^* \otimes \bigwedge^{f_0+2} F_1 \oplus \cdots] &\to R \otimes F_0
\end{align*}
By abuse of notation we will sometimes write $F_i$ to mean $F_i \otimes R$ when the meaning can be inferred from context. We write $w^{(i)}_j$ for the $j$-th component of the map $w^{(i)}$, with indexing starting at $j=0$. For instance $w^{(3)}_0 \colon F_3 \to F_2$ is just the differential $d_3$ of $\mb{F}$. Likewise $w^{(2)}_0 = d_2^*$ and $w^{(1)}_0 = d_1$. If $\mb{F}$ resolves $R/I$ for some ideal $I$ of grade at least\footnote{To be precise, it gives a choice of multiplicative structure on $0 \to F_3 \to F_2 \to F_1 \xto{a_2} R$ where $a_2$ comes from the First Structure Theorem of \cite{Buchsbaum-Eisenbud74}. If $I$ has grade at least 2, then it is equal to the image of $a_2$.} 2, then the maps $w^{(i)}_1$ give a choice of multiplicative structure on $\mb{F}$ lifting that on $\Tor_*(R/I,k)$. In general we will refer to the maps $w^{(i)}$ and their components $w^{(i)}_j$ as ``(higher) structure maps'' for the resolution $\mb{F}$.

The main technical result we will establish in \S\ref{sec:surjectivity-pf} is that the surjectivity of the maps $w^{(i)}$ is guaranteed if the Betti numbers of $R/I$ are Dynkin. This should be viewed as the substitute for the perfect pairing $F_1 \otimes F_2 \to F_3$ used in the Gorenstein case to identify $F_1 \cong F_2^*$, which gave surjectivity of the $n \times 2n$ matrix.

\begin{customthm}{\ref{thm:surjectivity}}
	Suppose that $\mb{F}$ is a resolution of Dynkin format over a $\mb{C}$-algebra $R$ such 
	that its dual $\mb{F}^*$ is also acyclic. Then if $w\colon \Rgen \to R$ specializes $\mb{F}^\mathrm{gen}$ to $\mb{F}$, the maps $w^{(i)}$ are surjective.
\end{customthm}
\begin{remark}
	Also of interest is the representation $W(a_2)$ generated by the entries of the Buchsbaum-Eisenbud multiplier $a_2$. The product of this representation with $a_1$ appears in the subrepresentation $\bigwedge^{f_0} F_0^* \otimes \bigwedge^{f_0}[F_1 \oplus \cdots]$ of $S_{f_0} W(d_1)$, which is to say the $f_0 \times f_0$ minors of $w^{(1)}$. We will show that the restriction of $w$ to $W(a_2)$ is surjective as well, but the importance of considering $W(a_2)$ is only apparent when dealing with module formats with $f_0 > 1$.
\end{remark}

In \S\ref{sec:background}, we provide background on Lie algebras and Schubert varieties. Afterwards, we summarize the key results pertaining to $\Rgen$ in \S\ref{bg:Rgen}, deferring technical proofs to Appendix~\ref{sec:Rgen-pfs}. Then we will prove Theorem~\ref{thm:surjectivity} in \S\ref{sec:surjectivity-pf}, and deduce some restrictions on graded Betti tables as a corollary. If $\mb{F}$ resolves $R/I$ for a grade three perfect ideal $I$ in a local ring $(R,\mf{m})$, the surjectivity of $w^{(1)}$ in Theorem~\ref{thm:surjectivity} is equivalent to the map being nonzero mod $\mf{m}$. Given this, we can ask for the first component of this structure map that is nonzero mod $\mf{m}$. After posing this question more precisely, we will show in \S\ref{sec:classify} that it has a well-defined answer for each ideal $I$, which may be used to classify the ideal. As an example in a very simple case, if $w^{(1)}_1\colon F_3^* \otimes \bigwedge^3 F_1 \to R$ is nonzero mod $\mf{m}$, then $I$ is a complete intersection. From the geometric perspective, this first nonzero component determines the stratum of the Schubert variety $X \subset G/P_{x_1}$ in which the map $\Spec k\to \Spec R \to G/P_{x_1}$ lands, and using this we show how the local defining equations of $X \subset G/P_{x_1}$ yield generic perfect ideals. We demonstrate in Example~\ref{ex:nonlicci-perfect} that Theorem~\ref{thm:surjectivity} fails without the Dynkin hypothesis. We briefly discuss the relationship to the theory of linkage in \S\ref{sec:beyond-ADE}, together with other directions for future study.

\subsection*{Acknowledgements}
This material is partially based upon work supported by the National Science Foundation under Grant No. DMS-1928930 and by the Alfred P. Sloan Foundation under grant G-2021-16778, while the authors were in residence at the Simons Laufer Mathematical Sciences Institute (formerly MSRI) in Berkeley, California, during the Spring 2024 semester. The first and third authors are supported by the grants MAESTRO NCN-UMO-2019/34/A/ST1/00263 - Research in Commutative Algebra and
Representation Theory, NAWA POWROTY - PPN/PPO/2018/1/00013/U/00001 - Applications of Lie algebras to Commutative Algebra, and OPUS grant National Science Centre, Poland grant UMO-2018/29/BST1/01290. The first author is also supported by the Miniatura grant
2023/07/X/ST1/01329 from NCN (Narodowe Centrum Nauki), which funded his visit to SLMath in April 2024.

The authors would like to thank Ela Celikbas, Lars Christensen, David Eisenbud, Sara Angela Filippini, Craig Huneke, Witold Kraskiewicz, Andrew Kustin, Jai Laxmi, Claudia Polini, Steven Sam, Jacinta Torres, Bernd Ulrich, and Oana Veliche for interesting discussions pertaining to this paper and related topics.

\section{Representation theory background}\label{sec:background}
The structure of the generic ring $\Rgen$ has important connections to representation theory. We provide the necessary background on Lie algebras and Schubert varieties in \S\ref{bg:lie} and \S\ref{bg:sch} respectively. For the former, we will work in the generality of Kac-Moody Lie algebras to better develop the theory of $\Rgen$ in \S\ref{bg:Rgen}. We refer the reader to \cite{Humphreys72} and \cite{Kumar02} for the material in this section.

\subsection{Lie algebras and representations}\label{bg:lie}

\subsubsection{Construction}\label{bg:lie-construction}
Fix integers $p,q,r \geq 1$, and let $T=T_{p,q,r}$ denote the graph
\[\begin{tikzcd}[column sep = small, row sep = small]
	x_{p-1} \ar[r,dash] & \cdots \ar[r,dash] & x_1 \ar[r,dash] & u \ar[r,dash]\ar[d,dash] & y_1 \ar[r,dash] & \cdots \ar[r,dash] & y_{q-1} \\
	&&& z_1 \ar[d,dash]\\
	&&& \vdots \ar[d,dash]\\
	&&& z_{r-1}
\end{tikzcd}\]
Let $n = p+q+r-2$ be the number of vertices. From the above graph, we construct an $n\times n$ matrix $A$, called the \emph{Cartan matrix}, whose rows and columns are indexed by the nodes of $T$:
\[
A = (a_{i,j})_{i,j \in T}, \quad a_{i,j} = \begin{cases}
	2 &\text{if $i = j$,}\\
	-1 &\text{if $i,j \in T$ are adjacent,}\\
	0 &\text{otherwise.}
\end{cases}
\]
$T$ is a Dynkin diagram if and only if $1/p + 1/q + 1/r > 1$; in this case we say it is of \emph{finite type}. For the applications to \S\ref{sec:surjectivity-pf} and beyond, we will only consider this case. We next describe how to construct the associated Lie algebra $\mf{g}$.

Let $\mf{h} = \mb{C}^{2n - \rank A}$, and pick independent sets $\Pi = \{\alpha_i\}_{i\in T} \subset \mf{h}^*$ and $\Pi^\vee = \{\alpha_i^\vee\}_{i\in T} \subset \mf{h}$ satisfying the condition
\[
\langle \alpha_i^\vee,\alpha_j \rangle = a_{i,j}.
\]
The $\alpha_i$ are the \emph{simple roots} and the $\alpha_i^\vee$ are the \emph{simple coroots}. If $1/p + 1/q + 1/r = 1$, then $T = E_{n-1}^{(1)}$ is of \emph{affine type} and $\rank A = n-1$. Otherwise $\rank A = n$, and $\Pi, \Pi^\vee$ are bases of $\mf{h}^*,\mf{h}$ respectively.

The Lie algebra $\mf{g} \coloneqq \mf{g}(T)$ is generated by $\mf{h}$ together with elements $e_i,f_i$ for $i\in T$, subject to the defining relations
\begin{gather*}
	[e_i,f_j] = \delta_{i,j} \alpha_i^\vee,\\
	[h,e_i] = \langle h, \alpha_i \rangle e_i, [h,f_i] = -\langle h,\alpha_i \rangle f_i  \text{ for } h \in \mf{h},\\
	[h,h'] = 0 \text{ for } h,h' \in \mf{h},\\
	\ad(e_i)^{1-a_{i,j}}(e_j) = \ad(f_i)^{1-a_{i,j}}(f_j) \text{ for } i \neq j.
\end{gather*}
Under the adjoint action of $\mf{h}$, the Lie algebra $\mf{g}$ decomposes into eigenspaces as $\mf{g} = \bigoplus \mf{g}_\alpha$, where
\[
\mf{g}_\alpha = \{x \in \mf{g} : [h,x] = \alpha(h)x \text{ for all } h \in \mf{h}\}.
\]
This is the \emph{root space decomposition} of $\mf{g}$.

For simplicity, we henceforth assume that $T$ is \emph{not} of affine type. Thus the Cartan matrix is invertible, and $\mf{g}$ is generated by $e_i$ and $f_i$ for $i \in T$ since $\alpha_i^\vee$ is a basis of $\mf{h}$. This is just for convenience of exposition; the theory we discuss remains valid in the affine case with minor adjustments.

\subsubsection{Gradings on $\mf{g}$}\label{bg:lie-grading1}

Let $Q \subset \mf{h}^*$ be the root lattice $\bigoplus_{i\in T} \mb{Z}\alpha_i$. If $\mf{g}_\alpha \neq 0$, then necessarily $\alpha \in Q$. If such an $\alpha$ is nonzero, we say it is a \emph{root}, and denote the set of all roots by $\Delta$. Hence the Lie algebra $\mf{g}$ is $Q$-graded. By singling out a vertex $t \in T$, this $Q$-grading can be coarsened to a $\mb{Z}$-grading by considering only the coefficient of $\alpha_t$. We refer to this as the $t$-grading. The sum of all $t$-gradings for $t\in T$ is called the \emph{principal gradation} on $\mf{g}$. The degree zero part in the principal gradation is the Cartan subalgebra $\mf{h}$. For $\alpha \in \Delta \cup \{0\}$, we write:
\begin{itemize}
	\item $\alpha >_t 0$ (resp. $\alpha \geq_t 0$) if the coefficient of $\alpha_t$ in $\alpha$ is positive (resp. nonnegative),
	\item $\alpha > 0$ (resp. $\alpha \geq 0$) if the coefficient of $\alpha_t$ in $\alpha$ is positive (resp. nonnegative) for some $t \in T$,
\end{itemize}
and similarly for $\alpha <_t 0, \alpha \leq_t 0, \alpha < 0, \alpha \leq 0$. We have $\Delta = \Delta^+ \amalg \Delta^-$ where $\Delta^+ = \{\alpha \in \Delta : \alpha > 0\}$ and $\Delta^- = \{\alpha\in \Delta : \alpha < 0\}$ are the sets of \emph{positive} and \emph{negative} roots respectively.

Using these notions, we define a few important subalgebras of $\mf{g}$:
\begin{align*}
	\mf{n}^+ &= \bigoplus_{\alpha > 0} \mf{g}_\alpha & \mf{n}^- &= \bigoplus_{\alpha < 0} \mf{g}_\alpha\\
	\mf{b}^+ &= \bigoplus_{\alpha \geq 0} \mf{g}_\alpha & \mf{b}^- &= \bigoplus_{\alpha \leq 0} \mf{g}_\alpha\\
	\mf{n}_t^+ &= \bigoplus_{\alpha >_t 0} \mf{g}_\alpha & \mf{n}_t^- &= \bigoplus_{\alpha <_t 0} \mf{g}_\alpha\\
	\mf{p}_t^+ &= \bigoplus_{\alpha \geq_t 0} \mf{g}_\alpha & \mf{p}_t^- &= \bigoplus_{\alpha \leq_t 0} \mf{g}_\alpha
\end{align*}
Write $h_i \in \mf{h}$ for the basis dual to the simple roots $\alpha_i \in \mf{h}^*$. The degree zero part of $\mf{g}$ in the $t$-grading is
\[
\mf{g}^{(t)} \times \mb{C}h_t
\]
where $\mf{g}^{(t)}$ is the subalgebra generated by $\{e_i, f_i\}_{i \neq t}$ and $\mb{C}h_t$ is the one-dimensional abelian Lie algebra spanned by $h_t$. The decomposition of $\mf{g}$ into $t$-graded components is just its decomposition into eigenspaces for the adjoint action of $h_t$:
\[
\mf{g} = \bigoplus_{j  \in \mb{Z}} \ker(\ad(h_t) - j).
\]

\begin{example}\label{ex:inclusion-of-sl}
	For the Dynkin diagram $A_n$, with vertices labeled as
	\[
	\begin{tikzcd}[column sep = small, row sep = small]
		1 \ar[r,dash] & 2 \ar[r,dash] & \cdots \ar[r,dash] & n,
	\end{tikzcd}
	\]
	the associated Lie algebra is $\mf{sl}_{n+1} \coloneqq \mf{sl}(\mb{C}^{n+1})$. Let $\epsilon_{ij}$ denote the $(n+1)\times(n+1)$ matrix whose entries are all 0 except for a single 1 in the $i$-th row and $j$-th column. Then it is customary to use the Lie algebra generators $e_i = \epsilon_{i,i+1}$ and $f_i = \epsilon_{i+1,i}$ for $i = 0,\ldots,n$.
	
	An ordered sequence of vertices $t_1,\ldots,t_n$ forming a subgraph $A_n \subset T$ yields an inclusion $\mf{sl}_{n+1} \hookrightarrow \mf{g}$ by sending the generators $e_i, f_i$ of $\mf{sl}_{n+1}$ to the corresponding elements $e_{t_i}, f_{t_i} \in \mf{g}$.
\end{example}
\begin{example}\label{ex:sl-subalgebras}
	Using the preceding, we identify a few particular subalgebras of $\mf{g}$ as follows, where $F_0 = \mb{C}^{p-1}$, $F_1 = \mb{C}^{p+q}$, $F_2 = \mb{C}^{q+r}$, and $F_3 = \mb{C}^{r-1}$:
	\begin{itemize}
		\item $\mf{sl}(F_0)$ corresponds to the ordered sequence of vertices $x_2,\ldots,x_{r_1}$,
		\item $\mf{sl}(F_1)$ corresponds to the ordered sequence of vertices $y_{r_2-2},\ldots,y_1,u,x_1,\ldots,x_{r_1}$,
		\item $\mf{sl}(F_2)$ corresponds to the ordered sequence of vertices $y_{r_2-2},\ldots,y_1,u,z_1,\ldots,z_{r_3}$, and
		\item $\mf{sl}(F_3)$ corresponds to the ordered sequence of vertices $z_2,\ldots,z_{r_3}$.
	\end{itemize}
	In particular, $\mf{g}^{(x_1)} = \mf{sl}(F_0) \times \mf{sl}(F_2)$ and $\mf{g}^{(z_1)} = \mf{sl}(F_1) \times \mf{sl}(F_3)$.
\end{example}

\subsubsection{Representations}\label{bg:reps}
Let $V$ be a representation of $\mf{g}$. For $\lambda \in \mf{h}$, define the \emph{$\lambda$-weight space of $V$} to be
\[
V_\lambda = \{v \in V : h v = \lambda(h)v \text{ for all }h \in \mf{h}\}.
\]
If $V_\lambda \neq 0$, then we say $\lambda$ is a \emph{weight} of $V$. A nonzero vector $v \in V_\lambda$ is a \emph{highest weight vector} if $\mf{n}^+ v = 0$. If such a $v$ generates $V$ as a $\mf{g}$-module, then we say $V$ is a \emph{highest weight module} with highest weight $\lambda$.

Let $\mc{U}$ denote the universal enveloping algebra functor. Representations of $\mf{g}$ are equivalent to modules over $\mc{U}(\mf{g})$. Given $\lambda \in \mf{h}^*$, the \emph{Verma module} $M(\lambda)$ is defined to be
\[
M(\lambda) = \mc{U}(\mf{g}) \otimes_{\mc{U}(\mf{b}^+)} \mb{C}_\lambda.
\]
Here $\mb{C}_\lambda$ is the $\mf{b}^+$-module where $\mf{h}$ acts by $\lambda$ and $\mf{n}^+$ acts trivially. All the weights of $M(\lambda)$ are in $\lambda + Q$. If $v \in V_\lambda$ is a highest weight vector, then there is a map $M(\lambda) \to V$ sending $1 \mapsto v$. If $V$ is a highest weight module then this map is surjective.

Every Verma module $M(\lambda)$ has a unique maximal proper submodule $J(\lambda)$. It follows that $L(\lambda) = M(\lambda)/J(\lambda)$ is an irreducible highest weight module with heighest weight $\lambda$, and any such module is isomorphic to $L(\lambda)$.

Let $\omega_i \in \mf{h}^*$ be the basis dual to $\alpha_i^\vee \in \mf{h}$. Explicitly, $\omega_i$ is the linear combination of $\alpha_i$ given by the $i$-th column of $A^{-1}$. These are the \emph{fundamental weights}, and the representations $L(\omega_i)$ are called \emph{fundamental representations}. Their nonnegative integral span is the collection of \emph{dominant weights}.

One can alternatively work with lowest weights instead of highest weights, interchanging the roles of positive and negative parts of the Lie algebra in all of the preceding. The irreducible representation with lowest weight $-\lambda$ is $L(\lambda)^\vee$, where $(-)^\vee$ represents the ``restricted'' dual. That is, $V^\vee \coloneqq \bigoplus V_\lambda^*$ for a weight module $V$ with finite-dimensional weight spaces. One has $V^\vee \subseteq V^*$, with equality when $V$ is finite-dimensional.

\subsubsection{Weight grading on representations}\label{sec:grading}
The decomposition of $L(\lambda)$ into weight spaces gives an $\mf{h}^*$-grading on $L(\lambda)$. Moreover, all the weights of $L(\lambda)$ are in the translate $\lambda + \bigoplus_{i\in T} \mb{Z} \alpha_i$ of the root lattice.

In \S\ref{bg:lie-grading1} it was described how singling out a vertex $t\in T$ allows us to impose a $\mb{Z}$-grading on $\mf{g}$ by considering only the coefficient of $\alpha_t$ in the $\mf{h}^*$-grading. This works for representations $L(\lambda)$ as well: if $v \in L(\lambda)$ is a highest weight vector then $h_t v = \langle h_t, \lambda \rangle v$ and the eigenvalues for the action of $h_t$ on $L(\lambda)$ are $\langle h_t, \lambda \rangle, \langle h_t, \lambda \rangle-1,\ldots$, terminating iff $L(\lambda)$ is finite-dimensional. The eigenspaces give the $t$-graded components. Each one is a representation of the subalgebra $\mf{g}^{(t)} \times \mb{C}h_t \subset \mf{g}$. In particular, $v$ is a highest weight vector for the top graded component, thus this component is the representation of $\mf{g}^{(t)}$ with highest weight $\sum_{i \neq t} c_i \omega_i$ if $\lambda = \sum_{i\in T} c_i \omega_i$. When we discuss the $t$-grading on a representation $L(\lambda)$ (resp. $L(\lambda)^\vee$), we will typically shift the indexing so that the top (resp. bottom) component is in $t$-degree zero.

\subsubsection{Exponential action and Baker-Campbell-Hausdorff}\label{bg:lie-exp}
Let $\bigoplus_{i>0} \mb{L}_i$ be a strictly positively graded Lie algebra, e.g. $\mf{n}_t^\pm$ for some $t\in T$. Its bracket naturally extends to one on $\mbf{L} = \prod_{i>0} \mb{L}_i$. Suppose $R$ is an $R_0$-algebra on which elements $X\in \mb{L}_i$ act by locally nilpotent $R_0$-linear derivations. Here ``locally nilpotent'' means that for any Lie algebra element $X$ and ring element $f \in R$, we have $X^N f = 0$ for $N \gg 0$. Then for any $X \in \mbf{L}$, the exponential
\[
\exp X = \Id + X + \frac{1}{2!} X^2 + \frac{1}{3!}X^3 + \cdots
\]
defines an $R_0$-algebra automorphism of $R$.

Moreover, given Lie algebra elements $X,Y \in \mbf{L}$, the Baker-Campbell-Hausdorff formula gives a well-defined element $Z \in \mbf{L}$ such that $\exp Z = \exp X \exp Y$:
\[
Z = X + Y + \frac{1}{2}[X,Y] + \frac{1}{12}[X,[X,Y]] - \frac{1}{12}[Y,[X,Y]] + \cdots.
\]
We will not need the explicit expression for $Z$, only that such an expression exists in terms of iterated commutators.

\subsection{Schubert varieties}\label{bg:sch}
For this section we assume that $T$ is a Dynkin diagram. There is a unique simply connected Lie group $G$ with Lie algebra $\mf{g}$. The representations of $G$ correspond to those of $\mf{g}$, and for a fundamental weight $\omega_t$, the action of $G$ on the highest weight line in $\mb{P}(L(\omega_t))$ has stabilizer $P_t^+$, the subgroup of $G$ corresponding to the maximal parabolic subalgebra $\mf{p}_t^+$ as defined in \S\ref{bg:lie-grading1}. Hence the orbit of this highest weight line can be identified with the homogeneous space $G/P_t^+$. For Dynkin type $A_n$ with the standard labeling of vertices, this produces the Grassmannian $\Gr(t,n+1)$. Accordingly, the reader not familiar with this construction may think of $G/P_t^+$ as a ``generalized Grassmannian.''

In the interest of only introducing what is necessary, we will not take this perspective and we will instead define all the objects we need completely algebraically. However, we will use notation which alludes to the geometric construction. For instance we will define the variety ``$G/P_t^+$'' via its homogeneous coordinate ring, without actually defining $G$, the subgroup $P_t^+$, or what it means to take a quotient. In a similar vein we will also define notions such as ``$P_t^-$-orbits,'' although we will include some remarks to motivate such terminology.

That being said, we will occasionally want to use legitimate group actions, but the group in each case will either be the special or general linear group, or the exponential of some Lie algebra whose elements have locally nilpotent actions (so that the exponential is algebraically well-defined).

\subsubsection{Definition of the homogeneous space $G/P$}\label{sec:GP-homog-ring}
Pick a vertex $t \in T$, let $\omega_t$ be the corresponding fundamental weight, and $L(\omega_t)$ the irreducible representation with highest weight $\omega_t$. Let $L(\omega_t)^\vee$ be the irreducible representation with lowest weight $-\omega_t$; since these representations are finite-dimensional we have $L(\omega_t)^* = L(\omega_t)^\vee$. Let $\mf{A}$ be the $\mb{C}$-algebra
\[
\bigoplus_{n\geq 0} L(n\omega_t)^\vee = (\Sym L(\omega_t)^\vee) / I_\Plucker
\]
where, viewing $\Sym L(\omega_t)^\vee$ as polynomial functions on $L(\omega_t)$, the ideal $I_\Plucker$ is comprised of subrepresentations vanishing on a highest weight vector $v \in L(\omega_t)$. As an ideal it is generated by quadrics, namely the irreducible representations in $S_2 L(\omega_t)^\vee$ other than $L(2\omega_t)^\vee$. These are the \emph{Pl\"ucker relations}.

The ring $\mf{A}$ is graded with $L(\omega_t)^\vee$ in degree 1. We define
\[
G/P_{t}^+ \coloneqq \operatorname{Proj} \mf{A} \subset \mb{P}(L(\omega_t)).
\]
\begin{remark}
	By taking just the $\mb{C}$-points of this scheme, one recovers the object described at the beginning of this section.
\end{remark} 

\subsubsection{Weyl group and subgroups}
Let $W$ denote the Weyl group associated to $T$. It is generated by the  \emph{simple reflections} $\{s_i\}_{t\in T}$. Explicitly,
\[
W = \langle \{s_i\}_{i\in T} \mid (s_i s_j)^{m_{ij}} = 1\rangle
\]
where $m_{ij} = 1$ if $i=j$, $m_{ij} = 2$ if $i,j \in T$ are not adjacent, and $m_{ij} = 3$ if $i,j \in T$ are adjacent. Since we are assuming $T$ to be a Dynkin diagram, $W$ is finite.

The Weyl group acts on $\mf{h}^*$: the simple reflections act via
\[
s_i(\lambda) = \lambda - \langle \alpha_i^\vee, \lambda\rangle \alpha_i.
\]
\begin{remark}\label{rem:Weyl-action-lift}
	It is also possible to lift (non-uniquely) each element $\sigma\in W$ to the group $G$. We can circumvent dealing with the group $G$ by noting that each individual simple reflection $s_t$ can be lifted to an element of the group $\SL_2$ corresponding to the Lie algebra $\mf{sl}_2$ spanned by $e_t, f_t, \alpha_t^\vee$. This group acts on $\mf{g}$ and its representations $L(\lambda)$ for $\lambda$ dominant. We can get an action of any $\sigma\in W$ by writing it as a product of simple reflections.
	
	We will often abuse notation and write e.g. $\sigma v$ with $\sigma\in W$ and $v \in L(\lambda)$ when we really mean to pick a lift of $\sigma$ and then act by that on $v$. One explicit choice lifting $s_t$ is $\exp(f_t) \exp (-e_t) \exp (f_t)$, though the particular choice will not matter for our purposes.
\end{remark}

The \emph{length} $\ell(\sigma)$ of an element $\sigma \in W$ is the minimum number of simple reflections needed to express $\sigma$ as a product of simple reflections. There is a unique longest element $w_0 \in W$.

If $t \in T$, we let $W_{P_t} \subset W$ denote the subgroup generated by all simple reflections other than $s_t$. Every coset $W/W_{P_t}$ has a minimal length representative, and the set of such representatives is denoted by $W^{P_t}$.

\subsubsection{Pl\"ucker coordinates and Schubert cells}\label{subsec:cells-P-orbits}
Let $v \in L(\omega_t)$ be a highest weight vector. By abuse of notation we also write $v \in \mb{P}(L(\omega_t))$ for its span. This is the \emph{Borel-fixed point} of $G/P_t^+$. The \emph{torus-fixed points} are $\sigma v$ for $\sigma \in W^{P_t}$. The $\mb{C}$-points of $G/P_t^+$ decompose into a disjoint union of \emph{Schubert cells}
\[
G/P_t^+ = \coprod_{\sigma \in W^{P_t}} C^\sigma
\]
where $C^\sigma$ is the orbit of $\sigma v$ under the action of $\exp \mf{n}^-$. For $j \in T$ and $[\sigma] \in W_{P_j} \backslash W / W_{P_t}$, we say that the union of Schubert cells
\[
\coprod_{\sigma' \in W^{P_t}, [\sigma']=[\sigma]} C^\sigma
\]
is a \emph{$P_{j}^-$-orbit}, where the union is taken over $\sigma'$ representing the same double coset as $\sigma$. See also Remark~\ref{rem:P-orbit}.

The variables of $\mf{A}$ are called \emph{Pl\"ucker coordinates}. Let $p_e$ denote a lowest weight vector of $L(\omega_t)^\vee$ and for $\sigma\in W^{P_t}$ let $p_\sigma = \sigma p_e$. The set $\{p_\sigma : \sigma\in W^{P_t}\}$ is the set of \emph{extremal Pl\"ucker coordinates}. The representation $L(\omega_t)^\vee$ may have weights other than those in the $W$-orbit of $-\omega_t$; these remaining Pl\"ucker coordinates are called non-extremal. The representation is \emph{miniscule} if all weights belong to the same $W$-orbit, in which case all Pl\"ucker coordinates are extremal.

The Pl\"ucker coordinate $p_\sigma$ vanishes on the Schubert cell $C^w$ iff $\sigma \not\geq w$ in the (strong) Bruhat order.

\subsubsection{The Schubert variety $X^w$}\label{bg:sch-schubert-var}
Using the same notation as above, the Schubert variety $X^w$ is the closure of the Schubert cell $C^w$. Let $\mf{n}^- w v$ denote the $\mf{n}^-$-representation generated by $w v$ inside of $L(\omega_t)$. This is a \emph{Demazure module}. Since Schubert varieties are linearly defined, the homogeneous defining ideal of $X^w$ is the ideal of $\mf{A}$ generated by the elements of $L(\omega_t)^\vee = L(\omega_t)^*$ which vanish on $\mf{n}^- w v$.

We will only be interested in a very specific Schubert variety, so we will now specialize the discussion accordingly. For the remainder of \S\ref{bg:sch}, we will always assume that $p=2$ and $t = x_1$, so that the distinguished node is the only node on the left arm. Let $w = s_{z_1} s_u s_{x_1}$; it is a minimal length representative for its coset in $W/W_{P_{x_1}}$. We define $X^w$ as above. The Pl\"ucker coordinates which vanish on $\mf{n}^- w v \subset L(\omega_{x_1})$ are exactly those in the bottom $z_1$-graded component $F_1 \subset L(\omega_{x_1})^\vee$ (c.f. Example~\ref{ex:sl-subalgebras}). Explicitly these are the following $q+2$ coordinates, all of which are extremal:
\[
p_e, p_{s_{x_1}}, p_{s_u s_{x_1}},\ldots, p_{s_{y_{q-1}} \cdots s_{y_1} s_u s_{x_1}}.
\]
All Schubert varieties are by construction preserved under the exponential action of $\mf{n}^-$, but this one is moreover preserved under the action of $\SL(F_1) \times \SL(F_3)$ corresponding to the subdiagram $T - \{z_1\}$. Hence it is preserved under the action of the semidirect product $\exp(\mf{n}_{z_1}^-)\rtimes (\SL(F_3)\times \SL(F_1))$.

\begin{remark}\label{rem:P-orbit}
	This semidirect product is closely related to the group $P_{z_1}^-$ whose definition we have omitted. Although they are not equal, they have the same orbits on $G/P_{x_1}^+$. The Schubert variety $X^w$ can also be described as the complement of the open $P_{z_1}^-$-orbit.
\end{remark}

\subsubsection{Affine patches}\label{bg:sch-affine-patches}
The scheme $G/P_t^+$ is covered by the open sets $\sigma C^e \coloneqq \{p_\sigma \neq 0\}$ for $\sigma \in W^{P_t}$. Each of these is isomorphic to affine space of dimension $\ell (w^{P_t}_0)$ where $w^{P_t}_0 \in W^{P_t}$ is a minimal length representative of $[w_0] \in W/W_{P_t}$. This is the dimension of $\mf{n}_t^-$, the negative part of $\mf{g}$ in the $t$-grading, and this algebra can be used to explicitly identify $\sigma C^e$ with affine space as follows.

Let $S$ be the polynomial ring $\Sym(\mf{n}_t^-)^*$ and let
\[
Z \in \mf{n}_t^- \otimes (\mf{n}_t^-)^* \subset \mf{n}_t^- \otimes S
\]
be adjoint to the identity on $\mf{n}_t^-$. We think of $Z$ as the ``generic element'' of $\mf{n}_t^-$. Then $\sigma C^e$ may be parametrized as
\[
S \xto{i_{x_1}^\mathrm{top}} S \otimes L(\omega_{x_1}) \xto{\sigma \exp Z} S\otimes L(\omega_{x_1})
\]
where $i_{x_1}^\mathrm{top}$ denotes the inclusion of the top $x_1$-graded component.

\subsubsection{The coordinate ring of $X^w$ restricted to an open cell}
Combining \S\ref{bg:sch-affine-patches} with \S\ref{bg:sch-schubert-var} and using the same notation, we get that the entries of the composite
\[
d_1^*\colon S \xto{i_{x_1}^\mathrm{top}} S \otimes L(\omega_{x_1}) \xto{\sigma \exp Z} S\otimes L(\omega_{x_1}) \xto{p_{z_1}^\mathrm{top}} S \otimes F_1^*
\]
give the equations cutting out $X^w$ restricted to $\sigma C^e$, where $p_{z_1}^\mathrm{top}$ denotes the projection onto the top $z_1$-graded component. Defining $d_1$ to be the dual of this composite, its cokernel is the coordinate ring of $X^w \cap \sigma C^e$. Moreover, it is shown in \cite{NW-examples} how to extend this to a free resolution: define $d_2$ to be
\[
d_2 \colon S \otimes F_2 \xto{i_{x_1}^\mathrm{top}} S \otimes L(\omega_{y_d}) \xto{\sigma \exp Z} S\otimes L(\omega_{y_d}) \xto{p_{z_1}^\mathrm{top}} S \otimes F_1
\]
and $d_3$ to be the dual of
\[
d_3^* \colon S \otimes F_2^* \xto{i_{x_1}^\mathrm{top}} S \otimes L(\omega_{z_t}) \xto{\sigma \exp Z} S\otimes L(\omega_{z_t}) \xto{p_{z_1}^\mathrm{top}} S \otimes F_3^*.
\]
\begin{thm}\label{thm:intermediate-res}
	The maps defined above assemble into a resolution of the coordinate ring of $X^w \cap \sigma C^e$:
	\[
	0 \to S \otimes F_3 \xto{d_3} S \otimes F_2 \xto{d_2} S \otimes F_1 \xto{d_1} S.
	\]
\end{thm}
\begin{proof}
	The paper \cite{NW-examples} actually only discussed the case $\sigma = w_0$, in which case applying $w_0$ and then projecting onto the top $z_1$-graded component is the same as projecting onto the bottom $z_1'$-graded component, where $z_1'$ is the node on $T$ ``dual'' to $z_1$ if $T$ has exceptional duality (c.f. \cite[Remark 2.2]{NW-examples}). That is the manner in which the differentials are presented in the referenced paper. But this was just a matter of perspective and not the result of any technical limitation, as we will discuss briefly around Theorem~\ref{thm:non-local-gen}. The proof of \cite[Lemma 2.4]{NW-examples} works exactly the same, showing that this is a complex. For acyclicity, the proof of \cite[Theorem 5.1]{NW-examples} shows how one can find powers of the Pl\"ucker coordinates among the minors of $d_i$. The only difference is that these are now written in terms of the affine coordinates on the patch $\sigma C^e$ instead of $w_0 C^e$, but that is inconsequential to the proof.
\end{proof}

Localizing this at the torus-fixed point $\sigma v\in \sigma C^e$, corresponding to the ideal of variables in $S$, gives a resolution of the local ring $\mc{O}_{X^w,\sigma v}$ over $\mc{O}_{G/P_t^+, \sigma v}$.

\section{The generic ring $\Rgen$}\label{bg:Rgen}
In this section, we introduce the machinery of higher structure maps and show how it can be used to produce a homomorphism $\mf{A} \to R$ (c.f. \S\ref{sec:GP-homog-ring}) given a grade 3 perfect ideal $I \subset R$ with Dynkin Betti numbers. In the next section we will show that the homomorphism $\mf{A} \to R$ determines a well-defined map $\Spec R \to G/P_{x_1}$ as alluded to in the introduction.

\begin{assumption}\label{ass:base-field}
	All rings throughout are assumed to be commutative $\mb{C}$-algebras. This streamlines the exposition regarding representation theory and Schubert varieties, but our main results in \S\ref{sec:classify} hold more generally over $\mb{Q}$ and we will state them accordingly. Indeed, the Lie algebra $\mf{g}$ from \S\ref{sec:background} may be defined using the same generators and relations over $\mb{Q}$, which results in the \emph{split form} of $\mf{g}$ whose representation theory parallels the situation over $\mb{C}$. The construction and decomposition of $\Rgen$ discussed in the present section remain valid.
	
	We will often make arguments using grade and regular sequences. These arguments remain valid in the non-Noetherian setting provided one uses Northcott's notion of ``true grade,'' adjoining additional indeterminates as needed. The arguments remain otherwise unaffected, so we will not belabor this technicality.
\end{assumption}

\subsection{Construction of $\Rgen$}\label{sec:Rgen:construction}
We first recall the construction of the generic pair $(\Rgen,\Fgen)$ from \cite{Weyman89}.

\subsubsection{Generic free resolutions}\label{bg:generic-free-res}
Let $\mb{F}$ be a complex of free modules
\[
\mb{F} \colon 0 \to F_c \to F_{c-1} \to \cdots \to F_1 \to F_0
\]
over some ring $R$. The \emph{format} of $F$ is the sequence of numbers $\underline{f} = (f_0,f_1,\ldots,f_c)$ where $f_i = \rank F_i$.

Let $S$ be a ring and $\mb{G}$ a resolution over $S$ with format $(f_0,\ldots,f_c)$. We say that $(S,\mb{G})$ is \emph{generic} for the given format if, given any resolution $\mb{F}$ of the same format over some ring $R$, there exists a map $\varphi\colon S \to R$ specializing $\mb{G}$ to $\mb{F}$. Note that we do not require the map $\varphi$ to be unique: Bruns showed that this is a mandatory concession for $c \geq 3$ \cite{Bruns84}. We will only be concerned with $c=3$ in this paper and we now restrict to this case.

Fix positive integers $r_1 \geq 1$, $r_2 \geq 2$, $r_3 \geq 1$ and define $f_0 = r_1$, $f_1 = r_1 + r_2$, $f_2 = r_2 + r_3$, $f_3 = r_3$. Let $\underline{f} = (f_0,f_1,f_2,f_3)$ and $F_i = \mb{C}^{f_i}$. For the format $\underline{f}$, a generic pair $(\Rgen(\underline{f}),\Fgen(\underline{f}))$ was constructed in \cite{Weyman89} and \cite{Weyman18}. Henceforth when we say \emph{the} generic ring or resolution, we refer to this construction specifically. We will typically suppress $\underline{f}$ from the notation. One can more generally consider formats where $f_0 \geq r_1$, but we will not do so here.

\subsubsection{The first structure theorem of Buchsbaum and Eisenbud}
The first structure theorem of \cite{Buchsbaum-Eisenbud74} is essential to the construction of $\Rgen$. We recall it here with two small adjustments: we state it only for $c=3$ and we avoid identifying top exterior powers with the base ring in the interest of doing things $\GL(F_i)$-equivariantly. We say that $\mb{F}$ is \emph{acyclic in grade $g$} if $\mb{F} \otimes R_\mf{p}$ is acyclic for all primes $\mf{p} \in \Spec R$ with $\grade \mf{p} \leq g$.

\begin{thm}\label{thm:BE1}
	Let $0 \to F_3 \xto{d_3} F_2 \xto{d_2} F_1 \xto{d_1} F_0$ be a complex of free modules, acyclic in grade 1, of format $(f_0,f_1,f_2,f_3)$ over $R$. Then there are uniquely determined maps $a_3,a_2,a_1$, constructed as follows:
	\begin{itemize}
		\item $a_3$ is the top exterior power
		\[
		a_3 \colon \bigwedge^{f_3} F_3 \to \bigwedge^{f_3} F_2.
		\]
		\item $a_2$ is the unique map making the following diagram commute:
		\[
		\begin{tikzcd}
			\bigwedge^{r_2} F_2 \ar[rr,"\bigwedge^{r_2}d_2"] \ar[dr,"-\wedge a_3",swap] && \bigwedge^{r_2} F_1\\
			& \bigwedge^{f_3} F_3^* \otimes \bigwedge^{f_2} F_2 \ar[ur,dashed,"a_2",swap]
		\end{tikzcd}
		\]
		\item $a_1$ is the unique map making the following diagram commute:
		\[
		\begin{tikzcd}
			\bigwedge^{f_0} F_1 \ar[rr,"\bigwedge^{r_1}d_1"] \ar[dr,"-\wedge a_2",swap] && \bigwedge^{f_0} F_0\\
			& \bigwedge^{f_3} F_3 \otimes \bigwedge^{f_2} F_2^* \otimes \bigwedge^{f_1} F_1 \ar[ur,dashed,"a_1",swap]
		\end{tikzcd}
		\]
	\end{itemize}
\end{thm}
Theorem~\ref{thm:BE1} can be used to construct a free complex $\mb{F}^a$ of the given format $\underline{f}$ over a ring $R_a$, called the Buchsbaum-Eisenbud multiplier ring, such that $\mb{F}^a$ is acyclic in grade 1 and is universal with respect to this property. In particular, if $\mb{F}$ is any resolution of the same format over some ring $R$, then there is a unique map $R_a \to R$ specializing $\mb{F}^a$ to $\mb{F}$.

However, $\mb{F}^a$ is not acyclic: letting $d_i$ denote the differentials of $\mb{F}^a$, we have
\[
\grade I_{r_3} (d_3) = 2, \quad \grade I_{r_2} (d_2) = 2, \quad \grade I_{r_1} (d_1) = 1.
\]
We recall the Buchsbaum-Eisenbud acyclicity criterion from \cite{Buchsbaum-Eisenbud73}: a finite free complex is acyclic exactly when $\grade I_{r_i}(d_i) \geq i$ and $f_i = r_{i+1} + r_i$ for all $i$. Hence the failure of $\mb{F}$ to be acyclic can be attributed to the insufficient grade of $I_{r_3} (d_3)$.

To remedy this, the strategy of \cite{Weyman89} is kill cycles in the Koszul complex on $\bigwedge^{r_3} d_3$. Explicitly, writing $\mc{K} = \bigwedge^{f_3} F_3^* \otimes \bigwedge^{f_3}F_2$ so that $\bigwedge^{r_3}d_3$ can be viewed as a map $R_a \to \mc{K}$, the fact that $\grade I_{r_3} (d_3) = 2$ implies the existence of nonzero $H^2$ in
\[
0 \to \bigwedge^0 \mc{K} \to \bigwedge^1 \mc{K} \to \bigwedge^2 \mc{K} \to \bigwedge^3 \mc{K}.
\]
Now the promised connection to representation theory gradually appears. First, in \cite{Weyman89}, the recursive procedure of killing $H^2$ in the complex above was carried out with the aid of a graded Lie algebra $\bigoplus_{i > 0} \mb{L}_i$ where we view $\mb{L}_i$ as residing in degree $-i$. Let
\[
\mb{L}^\vee \coloneqq \bigoplus_{i > 0} \mb{L}_i^*, \qquad \mbf{L} \coloneqq (\mb{L}^\vee)^* = \prod_{i > 0} \mb{L}_i.
\]
We call $\mbf{L}$ the \emph{defect Lie algebra}.

There is a diagram
\begin{equation*}\begin{tikzcd}
		0 \ar[r] & \bigwedge^0 \mc{K} \ar[r] & \bigwedge^1 \mc{K} \ar[r] & \bigwedge^2 \mc{K} \ar[r] & \bigwedge^3 \mc{K}\\
		&& \mb{L}^\vee \ar[u, "p"] \ar[r] & \bigwedge^2 \mb{L}^\vee \ar[u, "q"]
\end{tikzcd}\end{equation*}
where the dual of the lower horizontal map is the bracket in $\mbf{L}$. Let $p_i$ denote the restriction of $p$ to $\mb{L}_i^* \subset \mb{L}^\vee$ and similarly $q_i$ the restriction of $q = \bigwedge^2 p$ to $(\bigwedge^2 \mb{L}^\vee)_i$. The map $p_1$ is defined to lift a cycle constructed using the second structure theorem of \cite{Buchsbaum-Eisenbud74}. Since $\mb{L}^\vee$ is \emph{strictly} positively graded, $(\bigwedge^2 \mb{L}^\vee)_i$ only involves $\mb{L}_j$ for $j < i$, which allows for recursive computation of the cycles $q_m$ and their lifts $p_m$ for $m \geq 2$.

For positive integers $m$, define $R_m'$ to be the ring obtained from $R_a$ by generically adjoining variables for the entries of $p_1,\ldots,p_m$ and quotienting by all relations they would satisfy on a split exact complex (see for instance \cite[Lemma 2.4]{Weyman89}). Let $R_m$ be the ideal transform of $R_m'$ with respect to $I_{r_2}(d_2) I_{r_3}(d_3)$. The ring $\Rgen$ is defined to be the limit of the rings $R_m$, and $\Fgen \coloneqq \mb{F}^a \otimes \Rgen$.

The idea behind this construction is that adjoining the lifts $p_i$ ought to kill $H^2$ in the Koszul complex, and the ideal transform with respect to $I_{r_2}(d_2) I_{r_3}(d_3)$ ensures that we do not generate homology elsewhere in our complex. The acyclicity of $\Fgen$ was reduced to the exactness of certain 3-term complexes (c.f. \cite[Theorem 3.1]{Weyman89}), and this was later proven in \cite{Weyman18}.

\subsection{Exponential action of $\mbf{L}$}
Given a free resolution $\mb{F}$ over some ring $R$, a choice of maps $p_1,\ldots,p_m$ making
\begin{equation}\begin{tikzcd}\label{eq:p-lifting}
		0 \ar[r] & R \ar[r,"\bigwedge^{r_3} d_3"] & \mc{K}\otimes R \ar[r] & \bigwedge^2 \mc{K}\otimes R \ar[r] & \bigwedge^3 \mc{K}\otimes R\\
		&& \mb{L}_i^* \otimes R \ar[u, "p_i"] \ar[r] & (\bigwedge^2 \mb{L}^\vee)_i \otimes R \ar[u, "q_i"]
\end{tikzcd}\end{equation}
commute determines a map $R_m' \to R$. This extends uniquely to the ideal transform $R_m$ since the image of $I_{r_2}(d_2) I_{r_3}(d_3)$ in $R$ has grade at least 2.
\begin{lem}\label{lem:GFR-p-determines-w}
	There is a natural bijection
	\[
	\begin{tikzcd}
		\{\text{maps $w\colon \Rgen \to R$ specializing $\Fgen$ to $\mb{F}$}\} \ar[d,"\simeq",leftrightarrow]\\
		\{\text{choices of $\{p_i\}_{i > 0}$ making \eqref{eq:p-lifting} commute}\}
	\end{tikzcd}
	\]
\end{lem}
\begin{proof}
	This follows from the preceding discussion since $\Rgen = \lim R_m$.
\end{proof}
Furthermore, having chosen $p_i$ for $i < m$, the diagram \eqref{eq:p-lifting} shows that the non-uniqueness of $p_m$ lifting $q_m$ is exactly $\Hom(\mb{L}_m^* \otimes R,R) = \mb{L}_m \otimes R$. In \cite{Weyman89}, the action of $\mbf{L}$ on $\Rgen$ by derivations is described. Specifically, elements $u \in \mb{L}_n$ act on $\Rgen$ by $R_{n-1}$-linear derivations. It is sufficient to describe how they affect (the entries of) $p_{n+k}$ for $k \geq 0$, and this is as follows: the derivation $D_u$ sends $p_n^*$ to
\[
\mc{K}^* \xto{\bigwedge^{r_3} d_3^*} \Rgen \xto{u} \mb{L}_n \otimes \Rgen 
\]
and $p_{n+k}^*$ to
\[
\mc{K}^* \xto{p_k^*} \mb{L}_k \otimes \Rgen \xto{[u,-]} \mb{L}_{n+k} \otimes \Rgen.
\]
These are just restatements of the formulas given in \cite[Prop. 2.11]{Weyman89} and \cite[Thm. 2.12]{Weyman89} respectively.

These formulas naturally extend to an arbitrary element $X \in \mb{L} = \prod_{i > 0} \mb{L}_i$; the resulting derivation is well-defined because $\mb{L}_{>n}$ acts by zero on $R_n$. In a slight abuse of notation, we will also write $X$ for the corresponding derivation. Homomorphisms $\Rgen \to R$ correspond to $R$-algebra homomorphisms $\Rgen \otimes R \to R$, and the Lie algebra $\mb{L}\otimes R$ acts on $\Rgen \otimes R$.

For $X \in \mb{L} \otimes R$, the action of $\exp X \coloneqq \sum_{i\geq 0} \frac{1}{i!}X^i$ on $\Rgen \otimes R$ is well-defined since every element of $\Rgen \otimes R$ is killed by a sufficiently high power of $X$. Since $X$ acts by an $(R_a\otimes R)$-linear derivation, it follows formally that $\exp X$ acts by an automorphism fixing $R_a\otimes R$. As the next result shows, such automorphisms completely describe the non-uniqueness of the map $\Rgen \to R$ given a particular resolution $(R,\mb{F})$. We defer its proof to Appendix~\ref{sec:Rgen-pfs}.

\begin{customthm}{\ref{thm:parametrize}}
	Let $\mb{F}$ be a resolution of length three over $R$ and let $\Rgen$ be the generic ring for the associated format. Fix a $\mb{C}$-algebra homomorphism $w\colon \Rgen \to R$ specializing $\mb{F}^\mathrm{gen}$ to $\mb{F}$. Then $w$ determines a bijection
	\[
	\mbf{L}\cotimes R \coloneqq \prod_{i > 0} (\mb{L}_i \otimes R) \simeq \{\text{$\mb{C}$-algebra homomorphisms $w'\colon \Rgen \to R$ specializing $\mb{F}^\mathrm{gen}$ to $\mb{F}$}\}.
	\]
	Note that a $\mb{C}$-algebra homomorphism $\Rgen \to R$ can be viewed as an $R$-algebra homomorphism $\Rgen \otimes R\to R$. The correspondence above identifies $X \in \mbf{L} \cotimes R$ with the map $w\exp X$ obtained by precomposing $w$ with the action of $\exp X$ on $\Rgen\otimes R$. 
\end{customthm}

\subsection{The critical representations}\label{sec:crit-reps}

Recall that we have fixed parameters $r_1 \geq 1$, $r_2 \geq 2$, $r_3 \geq 1$, from which our format $\underline{f}$ is defined as $(f_0,f_1,f_2,f_3) = (r_1,r_1+r_2,r_2+r_3,r_3)$. Consider the diagram $T = T_{r_1+1,r_2-1,r_3+1}$
\[
\begin{tikzcd}[column sep=small, row sep=small]
	x_{r_1} \ar[r,dash] & \cdots \ar[r,dash] & \colorX{x_1} \ar[r,dash] & u\ar[r,dash]\ar[d,dash] & y_1\ar[r,dash] &\ar[r,dash] \cdots\ar[r,dash] & y_{r_2-2}\\
	&&& \colorZ{z_1}\ar[d,dash]\\
	&&& \vdots\ar[d,dash]\\
	&&& z_{r_3}
\end{tikzcd}
\]
and let $\mf{g}$ denote the associated Kac-Moody Lie algebra. Following Example~\ref{ex:sl-subalgebras}, we view the Lie algebras $\mf{sl}(F_j)$ as subalgebras of $\mf{g}$.

One of the main observations driving the results of \cite{Weyman18} is that
\begin{align*}
	\mb{L}^\vee = \mf{n}_{z_1}^+ = \bigoplus_{\alpha >_{z_1} 0} \mf{g}_\alpha, \qquad \mbf{L} = \hat{\mf{n}}_{z_1}^- = \prod_{\alpha <_{z_1} 0} \mf{g}_\alpha,
\end{align*}
and that the action of $\bigoplus \mb{L}_i$ on $\Rgen$ extends to an action of $\mf{g}$. We refer to \S\ref{bg:lie-grading1} for explanations regarding notation here.

The decomposition of $\Rgen$ into representations for the product $\mf{sl}(F_0) \times \mf{sl}(F_2) \times \mf{g}$ is detailed in \cite{Weyman18}. Of these representations, there are a few of particular interest, which we call the \emph{critical representations}---they are the ones generated by the entries of the differentials $d_i$ and Buchsbaum-Eisenbud multipliers $a_i$ for $\Fgen$. We denote these representations by $W(d_i)$ and $W(a_i)$ respectively. Let $L(\omega)$ be the irreducible representation with highest weight $\omega$ so that $L(\omega)^\vee$ is the irreducible representation with lowest weight $-\omega$. The aforementioned representations are
\begin{align}
	\begin{split}\label{eq:crit-reps}
		W(d_3) &= F_2^* \otimes L(\omega_{z_{r-1}})^\vee\\
		&= F_2^* \otimes [F_3 \oplus M^* \otimes\bigwedge^{f_0+1} F_1 \oplus \cdots]\\
		W(d_2) &= F_2 \otimes L(\omega_{y_{q-1}})^\vee\\
		&= F_2 \otimes [F_1^* \oplus M^* \otimes F_3^* \otimes \bigwedge^{f_0} F_1 \oplus  \cdots]\\
		W(d_1) &= F_0^* \otimes L(\omega_{x_{p-1}})^\vee\\
		&= F_0^* \otimes [F_1 \oplus M^* \otimes F_3^* \otimes \bigwedge^{f_0+2} F_1 \oplus \cdots]\\
		W(a_3) &= \bigwedge^{f_3} F_2^* \otimes L(\omega_{z_{1}})^\vee\\
		&= \bigwedge^{f_3} F_2^* \otimes [\bigwedge^{f_3} F_3 \oplus \cdots]\\
		W(a_2) &= \bigwedge^{f_2} F_2 \otimes L(\omega_{x_1})^\vee\\
		&= \bigwedge^{f_2} F_2 \otimes [\bigwedge^{r_2}F_1^* \otimes \bigwedge^{f_3} F_3^* \oplus \cdots]\\
		W(a_1) &= \bigwedge^{f_0} F_0^* \otimes \bigwedge^{f_1} F_1 \otimes \bigwedge^{f_2} F_2^* \otimes \bigwedge^{f_3} F_3.
	\end{split}
\end{align}
\begin{remark}\label{rem:gl-vs-sl}
	While the ring $\Rgen$ certainly has an action of $\prod \mf{gl}(F_i)$ by construction, reflected in the decomposition displayed above, this action does \emph{not} come from an inclusion of $\prod \mf{gl}(F_i)$ into $\mf{sl}(F_0)\times \mf{sl}(F_2) \times \mf{g}$. Nor is it correct to say that $\mf{gl}(F_0) \times \mf{gl}(F_2) \times \mf{g}$ acts on $\Rgen$, since the action of $\mf{gl}(F_2)$ does not commute with the action of $\mf{g}$. Rather, if we let
	\[
	M = \bigwedge^{f_3} F_3 \otimes \bigwedge^{f_2} F_2^* \otimes \bigwedge^{f_1} F_1
	\]
	then the $\prod \mf{gl}(F_i)$-equivariant description of $\mb{L}_1$ from \cite{Weyman89} is
	\[
	\mb{L}_1 = \bigwedge^{r_1+1} F_1^* \otimes F_3 \otimes M.
	\]
	Given an irreducible lowest weight representation $L(\omega)^\vee$ of $\mf{g}$, the action of $\prod \mf{gl}(F_i)$ on $L(\omega)^\vee$ can be inferred from its action on any $z_1$-graded component, e.g. the bottom one.
	
	Furthermore, if $\grade I_{r_0}(d_0) \geq 2$, then the Buchsbaum-Eisenbud multiplier $a_1\colon M \to \bigwedge^{f_0} F_0$ from Theorem~\ref{thm:BE1} is an isomorphism, and we may instead view
	\[
	\mb{L}_1 = \bigwedge^{r_1+1} F_1^* \otimes F_3 \otimes \bigwedge^{f_0} F_0
	\]
	which is sometimes more convenient, especially when dealing with resolutions of cyclic modules.
\end{remark}
\begin{remark}\label{rem:a3-in-d3}
	The representation $W(a_3)$ is contained in $S_{r_3}W(d_3)$:
	\[
	W(a_3) = \bigwedge^{f_3} F_2^* \otimes L(\omega_{z_1})^\vee \subseteq \bigwedge^{f_3} F_2^* \otimes \bigwedge^{f_3} L(\omega_{z_{r_3}})^\vee \subseteq S_{r_3} W(d_3).
	\]
\end{remark}

\subsection{Higher structure maps}\label{sec:GFR-HST}
Given a map $w\colon \Rgen \to R$ for a complex $(R,\mb{F})$, we denote by $w^{(i)}$ the restriction of $w$ to the representation $W(d_i) \subset R_\mathrm{gen}$ and $w^{(a_i)}$ the restriction of $w$ to the representation $W(a_i)$. We typically view these maps as being $R$-linear with source $L(\omega)^\vee \otimes R$, e.g. we think of $w^{(3)}$ as an $R$-linear map
\[
w^{(3)}\colon L(\omega_{z_{r_3}})^\vee \otimes R = [F_3 \oplus M^* \otimes\bigwedge^{f_0+1} F_1 \oplus \cdots] \otimes R \to F_2\otimes R
\]
as opposed to a $\mb{C}$-linear map $F_2^* \otimes L(\omega_{z_{r_3}})^\vee \to R$.

\begin{remark}\label{rem:a3-in-d3'}
	From Remark~\ref{rem:a3-in-d3}, we see that the entries of $w^{(a_3)}$ are given by certain $r_3 \times r_3$ minors of $w^{(3)}$. In bottom degree for example, we recover that $a_3 = \bigwedge^{r_3} d_3$.
\end{remark}

The homomorphism $w$ can be recovered from the three restrictions $w^{(i)}$; see Lemma~\ref{lem:V-exp-props} and Lemma~\ref{prop:localization-lemma}. We use $w^{(*)}_j$ to denote the restriction of $w^{(*)}$ to the $j$-th $z_1$-graded component of the representation, indexed so that $j=0$ corresponds to the bottom graded piece. For instance, $w^{(i)}_0 = d_i$ for $i = 1,3$ and $w^{(2)}_0 = d_2^*$. We call the maps $w^{(*)}_{>0}$ (a specific choice of) \emph{higher structure maps} for $\mb{F}$. Let us demonstrate Theorem~\ref{thm:parametrize} from this perspective.
\begin{example}\label{ex:multiplication}
	Consider a free resolution $\mb{F}$ of format $(1,f_1,f_2,f_3)$ resolving $R/I$ where $\operatorname{depth} I \geq 2$, and make the identification described in Remark~\ref{rem:gl-vs-sl} using $a_1$. The structure maps $w^{(i)}_1$ give a choice of multiplicative structure on $\mb{F}$; see \cite[Prop. 7.1]{Lee-Weyman19}. Explicitly, such a resolution has the (non-unique) structure of a commutative differential graded algebra, and the non-uniqueness is evidently seen from the fact that the multiplication $\bigwedge^2 F_1 \to F_2$ may be chosen as any lift in the diagram
	\begin{equation}\label{eq:example-w31}\begin{tikzcd}
			0 \ar[r] & F_3 \ar[r] & F_2 \ar[r] & F_1 \ar[r] & R\\
			&&& \bigwedge^2 F_1 \ar[ul, dashed] \ar[u]
	\end{tikzcd}\end{equation}
	where the map $\bigwedge^2 F_1 \to F_1$ is given by $e_1 \wedge e_2 \mapsto d_1(e_1) e_2 - d_1(e_2) e_1$. Indeed, we have that $\mb{L}_1 = F_3 \otimes \bigwedge^2 F_1^*$, which is exactly the non-uniqueness witnessed here.
	
	Now suppose that $w\colon R_\mathrm{gen} \to R$ (equivalently, $R \otimes R_\mathrm{gen} \to R$) is one choice of higher structure maps for $\mb{F}$, and take an element $X = \sum_{i > 0} u_i \in \mb{L} \otimes R$ using the same notation as before. Let $w' = w\exp(X)$, i.e.
	\[
	w' = w\left(1 + u_1 + \left(\frac{1}{2}u_1^2 + u_2\right) + \cdots\right)
	\]
	Note that $u_k$ maps $W(d_i)_{j}$ to $W(d_i)_{j-k}$. If we restrict the above equation to the representation $W(d_3)$ and expand it degree-wise, we get
	\begin{align*}
		w'^{(3)}_0 &= w^{(3)}_0\\
		w'^{(3)}_1 &= w^{(3)}_1 + w^{(3)}_0 u_1\\
		w'^{(3)}_2 &= w^{(3)}_2 + w^{(3)}_1 u_1 + w^{(3)}_0 \left(\frac{1}{2}u_1^2 + u_2\right)\\
		\vdots
	\end{align*}
	The first equation reflects that the underlying complex is still the same $\mb{F}$. The next equation shows that the new multiplication, viewed as a map $F_2^* \otimes \bigwedge^2 F_1 \to R$, was obtained from the old one by adding the composite
	\[
	F_2^* \otimes \bigwedge^2 F_1 \xto{1 \otimes u_1} F_2^* \otimes F_3 \xto{d_3} R.
	\]
	Here $u_1 \in \mb{L}_1 = F_3 \otimes \bigwedge^2 F_1^*$ could've been any map $\bigwedge^2 F_1 \to F_3$, and this exactly matches what we see in \eqref{eq:example-w31}.
\end{example}

While we are able to relate different choices of $w\colon \Rgen \to R$ specializing $\Fgen$ to a given $\mb{F}$, we have not discussed how to effectively compute a particular choice of $w$ given a resolution $\mb{F}$ in the first place. In principle, the maps $p_m$ may be computed recursively using the diagram \eqref{eq:p-lifting}, but it is less clear how to compute the maps $w^{(i)}$. For this, the following observation will be helpful.
\begin{prop}\label{prop:equivariant-p}
	Let $R$ be a ring and $\mb{F}$ a free resolution
	\[
	0 \to F_3 \otimes R \to F_2 \otimes R \to F_1 \otimes R \to F_0 \otimes R.
	\]
	\begin{enumerate}
		\item Suppose a group $G$ acts on the free modules $F_i \otimes R$ and the differentials of $\mb{F}$ are $G$-equivariant. Since $\prod \GL(F_i \otimes R)$ acts on $\Rgen \otimes R$ and on $\mb{L}_m^*$, we have an induced action of $G$ on $\Rgen \otimes R$ and $\mb{L}_m^*$. If the maps $p_m$ in \eqref{eq:p-lifting} are chosen to be $G$-equivariant, then the induced map $w\colon \Rgen \otimes R \to R$ is also $G$-equivariant.
		\item Suppose $R$ is graded and $\mb{F}$ is a graded free resolution where the differentials are homogeneous of degree zero. This induces a grading on $\Rgen \otimes R$ and $\mb{L}_m^*$ for each $m$, and it is possible to choose all $p_m$ in \eqref{eq:p-lifting} to be homogeneous of degree zero. The corresponding $w \colon \Rgen\otimes R \to R$ is then also homogeneous of degree zero.
	\end{enumerate}
\end{prop}
The construction of the cycle $q_1$ in terms of the differentials $d_i$ and the construction of $q_m$ in terms of $p_1,\ldots,p_{m-1}$ are both $\prod \GL(F_i)$-equivariant. So in the setting of Proposition~\ref{prop:equivariant-p}, (1) the cycle $q_m$ is $G$-equivariant provided $p_1,\ldots,p_{m-1}$ are, and similarly (2) $q_m$ is homogeneous of degree zero provided $p_1,\ldots,p_{m-1}$ are.
\begin{proof}
	This is evident from the construction of $\Rgen = \lim R_m$ and Lemma~\ref{lem:GFR-p-determines-w}. In the graded setting, it is always possible to recursively take $p_m$ which is homogeneous of degree zero: simply take any lift of $q_m$ (which is homogeneous of degree zero by induction) and discard the components which are not homogeneous of degree zero.
\end{proof}

\subsection{Computing particular higher structure maps}
We now demonstrate how, with the presence of some symmetry, Proposition~\ref{prop:equivariant-p} can assist in determining a particular choice of structure maps $w^{(i)}$ for $\mb{F}$.

\begin{example}\label{ex:non-perfect}
	Let $R = \mb{C}[t_1,t_2,t_3,t_4]$ with the standard $\mb{Z}$-grading and let $I = (t_1 t_2, t_2 t_3, t_3 t_4, t_4 t_1) \subset R$. The minimal graded free resolution of $R/I$ is
	\[
	\mb{F} \colon 0\to R(-4) \to R^4(-3) \to R^4(-2) \to R.
	\]
	As per Proposition~\ref{prop:equivariant-p}, there is a choice of $w\colon\Rgen \otimes R\to R$ specializing $\Fgen$ to $\mb{F}$ that is homogeneous of degree zero. We identify $M\cong R$ as in Remark~\ref{rem:gl-vs-sl}. The structure maps have the form
	\begin{align*}
		w^{(1)} \colon [F_1 \oplus \bigwedge^3 F_1]\otimes R &\to R\\
		w^{(2)} \colon [F_1^* \oplus F_1 \otimes F_3^*] \otimes R &\to F_2^* \otimes R\\
		w^{(3)} \colon [F_3 \oplus \bigwedge^2 F_1 \oplus F_3^* \otimes \bigwedge^4 F_1]\otimes R &\to F_2 \otimes R
	\end{align*}
	In this example $\mb{L}_1 = \bigwedge^2 F_1^* \otimes F_3$ is concentrated in degree zero, causing all generators of each representation $L(\omega)^\vee$ to be in the same degree. So with a homogeneous choice of $w$, all entries of $w^{(1)}$ have degree 2, whereas all entries of $w^{(2)}$ and $w^{(3)}$ are linear. 
\end{example}

\begin{example}\label{ex:d3-split}
	Let $\mb{F}$ be a resolution where $d_3$ is a split inclusion. After a change of coordinates, we assume it has the form
	\[
	0 \to F_3 \otimes R \to (F_3 \oplus Z)\otimes R \to F_1 \otimes R \to F_0 \otimes R
	\]
	where $Z = \mb{C}^{r_2}$ and $d_3$ maps $F_3 \otimes R$ identically to itself.
	
	There is an action of $G = \GL(F_3 \otimes R)$ on this resolution, and the differentials are equivariant with respect to it. The Koszul complex on $\bigwedge^{r_3}(d_3)$ is just a split exact complex, so we can certainly pick lifts $p_m$ which are $G$-equivariant, e.g. by using a $G$-equivariant splitting.
	
	Let $w \colon \Rgen \otimes R \to R$ be the $G$-equivariant map obtained in this manner. The maps $w^{(i)}$ have the form
	\begin{align*}
		[F_3 \oplus \bigwedge^{f_0+1} F_1 \otimes M^* \oplus \cdots] \otimes R &\to (F_3 \oplus Z) \otimes R\\
		[F_1^* \oplus F_3^* \otimes \bigwedge^{f_0}F_1 \otimes M^* \oplus \cdots]\otimes R &\to (F_3^* \oplus Z^*) \otimes R\\
		[F_1 \oplus F_3^* \otimes \bigwedge^{f_0+2} F_1 \otimes M^* \oplus \cdots] \otimes R &\to F_0 \otimes R
	\end{align*}
	Note that $M = \bigwedge^{f_3} F_3 \otimes \bigwedge^{f_2}(F_3^* \oplus Z^*) \otimes \bigwedge^{f_1} F_1$ is a trivial representation of $G$. Since $\mb{L}_1^* = \bigwedge^2 F_1 \otimes F_3^* \otimes M^*$, we accumulate an additional factor of $F_3^*$ every time we go up in the $z_1$-grading in each representation. Thus by $G$-equivariance considerations the only components $w^{(i)}_j$ that have any chance of being nonzero are:
	\begin{align*}
		w^{(3)}_1 \colon \bigwedge^{f_0+1}F_1 \otimes M^* \otimes R &\to Z \otimes R \subset (F_3 \oplus Z )\otimes R\\
		w^{(2)}_1 \colon \bigwedge^{f_0} F_1 \otimes F_3^* \otimes M^* \otimes R&\to F_3^* \otimes R \subset (F_3^* \oplus Z^*) \otimes R
	\end{align*}
	in addition to the maps $w^{(i)}_0$ which are obviously nonzero since they give the differentials of the resolution.
	
	We also note that for the map $w^{(a_2)}$
	\[
	[\bigwedge^{r_2} F_1^* \otimes \bigwedge^{f_3} F_3^* \oplus \cdots] \otimes R \to \bigwedge^{r_2} Z^* \otimes \bigwedge^{r_3} F_3^* \otimes R
	\]
	only the bottom component, namely $a_2$ itself, is nonzero by the same considerations.
\end{example}
Hence it would be beneficial to at least understand how to compute $w^{(3)}_1$ and $w^{(2)}_1$ explicitly. In \cite[Prop. 7.1]{Lee-Weyman19}, it is described\footnote{In the referenced paper, it was assumed that $a_1\colon M \to \bigwedge^{f_0} F_0^*$ is an isomorphism.} how to compute these maps via a comparison map from a Buchsbaum-Rim complex, which we now recall. We write $F_i$ to mean $F_i \otimes R$ in the following. Theorem~\ref{thm:BE1} gives a factorization
\[
\begin{tikzcd}
	\bigwedge^{r_1} F_1 \ar[rr, "\bigwedge^{r_1}d_1"]\ar[dr] &&\bigwedge^{r_1} F_0\\
	& M \ar[ur, "a_1", swap]
\end{tikzcd}
\]
in particular a map $\beta\colon M^* \otimes \bigwedge^{r_1}F_1 \to R$, which is essentially $a_2^*$ after appropriate identifications. It is straightforward to check that the composite
\[
M^* \otimes \bigwedge^{r_1+1}F_1 \to M^* \otimes \bigwedge^{r_1} F_1 \otimes F_1 \xto{\beta \otimes 1} F_1 \xto{d_1} F_0
\]
is zero, thus we can lift through $d_2$ to obtain a map
\[
w^{(3)}_1 \colon M^*\otimes \bigwedge^{r_1+1}F_1 \to F_2.
\]
The difference of the two maps
\begin{gather*}
	M^* \otimes \bigwedge^{r_1}F_1 \otimes F_2 \xto{\beta \otimes 1} F_2 \\
	M^* \otimes \bigwedge^{r_1}F_1 \otimes F_2 \xto{1 \otimes d_2} M^* \otimes \bigwedge^{r_1}F_1 \otimes F_1 \to M^* \otimes \bigwedge^{r_1+1} F_1 \xto{w^{(3)}_1} F_2
\end{gather*}
has image landing in $\ker d_2$, and thus it can be lifted through $d_3$ to obtain
\[
w^{(2)}_1 \colon M^* \otimes \bigwedge^{r_1}F_1 \otimes F_2 \to F_3.
\]
In the case that $r_0 = 1$, these maps can be viewed as giving a choice of multiplication on the resolution
\[
0 \to M^* \otimes F_3 \to M^* \otimes F_2 \to M^* \otimes F_1 \xto{\beta} R
\]
recovering what was illustrated in Example~\ref{ex:multiplication}.

A very important special case of Example~\ref{ex:d3-split} is when the entire complex $\mb{F}$ is split exact, e.g. we take $\mb{F}$ to be the split exact complex
\[
\mb{F}^\ssc \colon 0 \to F_3 \to F_3 \oplus Z \to F_0 \oplus Z \to F_0
\]
of $\mb{C}$-vector spaces. Here $M = \bigwedge^{f_0} F_0$, and a direct computation shows that there is a unique $G = \GL(F_0) \times \GL(F_3) \times \GL(Z)$-equivariant choice of $w$. For this choice of $w$, direct computation with the explicit definitions of $w^{(3)}_1$ and $w^{(2)}_1$ above shows that
\[
w^{(3)}_1 \colon \bigwedge^{f_0+1} F_1 \otimes \bigwedge^{f_0} F_0^* = Z \oplus F_0^* \otimes \bigwedge^2 Z \oplus \cdots  \to F_3 \oplus Z
\]
maps $Z$ identically to itself and is zero on all other factors by $G$-equivariance. Similarly
\[
w^{(2)}_1 \colon \bigwedge^{f_0} F_1 \otimes F_3^* \otimes \bigwedge^{f_0} F_0^* = (\mb{C} \oplus F_0^* \otimes Z \oplus \bigwedge^2 F_0^* \otimes \bigwedge^2 Z \oplus \cdots)\otimes F_3^* \to F_3^* \oplus Z^* = F_2^*
\]
maps $F_3^*$ identically to itself and is zero on all other factors.

If we view $\mf{sl}(F_0) \times \mf{sl}(F_2)$ as the subalgebra $\mf{g}^{(x_1)}$ of $\mf{g}$ as described at the beginning of \S\ref{sec:crit-reps}, we conclude:
\begin{thm}\label{thm:ssc}
	There exists a $\mb{C}$-algebra homomorphism $w_\mathrm{ssc}\colon \Rgen \to \mb{C}$ so that
	\begin{align*}
		w_\ssc^{(1)} \colon L(\omega_{x_{r_1}}) &\twoheadrightarrow F_0\\
		w_\ssc^{(2)} \colon L(\omega_{y_{r_2-2}}) &\twoheadrightarrow F_2^*\\
		w_\ssc^{(3)} \colon L(\omega_{z_{r_3}}) &\twoheadrightarrow F_2
	\end{align*}
	are given by projection onto the bottom $x_1$-graded component. Note that this determines $w_\ssc$ completely by Proposition~\ref{prop:HST-determined-by-a3}. Furthermore, $w^{(a_2)}$ is also projection onto its bottom $x_1$-graded component, which is its lowest weight space.
	
	The map $w_\ssc$ specializes $\Fgen$ to the \emph{standard split complex} of $\mb{C}$-vector spaces
	\[
	\mb{F}^\ssc \colon 0 \to F_3 \to F_2 \oplus Z \to F_0 \oplus Z \to F_0
	\]
	where $Z = \mb{C}^{r_2}$.
\end{thm}

The comment regarding $w^{(a_2)}$ in this theorem implies the following.
\begin{cor}\label{cor:GP-homogeneous-ring-wa2}
	Let $w\colon \Rgen \to R$ specialize $\Fgen$ to a resolution $\mb{F}$ of Dynkin format. Then there is a unique ring homomorphism
	\[
	\mf{A} = \bigoplus_{n \geq 0} L(n\omega_{x_1})^\vee \to R
	\]
	equal to $w^{(a_2)}$ in degree 1. Here the source is the homogeneous coordinate ring of $G/P = G/P_{x_1}$.
\end{cor}
Note that if $\mb{F}$ resolves a cyclic module (i.e. $r_1 = 1$), then $w^{(a_2)}$ is simply $w^{(1)}$ divided by the scalar $a_1$.

\begin{proof}
	The homogeneous coordinate ring of $G/P$ is generated in degree 1 so the uniqueness is clear; we need to check that the map is well-defined. As noted in Theorem~\ref{thm:ssc}, this is certainly true for $w = w_\ssc$, so the result for split $\mb{F}$ follows from Theorem~\ref{thm:parametrize} since $\GL(F_i)$ and $\exp \mbf{L}$ both act on the homogeneous coordinate ring of $G/P$.
	
	The ring $\bigoplus_{n \geq 0} L(n\omega_{x_1})^\vee$ is a quotient of $\Sym L(n\omega_{x_1})^\vee$ by Pl\"ucker relations. An arbitrary $\mb{F}$ is split after localization, and relations which hold over the localization must also hold over $R$.
\end{proof}

For this homomorphism $\mf{A} \to R$ to describe a well-defined map $\Spec R \to G/P$, it is necessary for $w^{(a_2)}$ to be surjective. This is \emph{not} automatic: in Example~\ref{ex:non-perfect}, all entries of the matrix $w^{(a_2)}$ (identified with $w^{(1)}$ by $M \cong R$) are in $\mf{m}$, thus $w^{(a_2)}$ is not surjective.

Stated geometrically, it is clear from the above that we at least obtain a map $\Spec R \setminus V(H_0(\mb{F})) \to G/P$, since $\mb{F}$ is split on the complement of $V(H_0(\mb{F}))$. We would like to argue that, under appropriate hypotheses, this map extends to the entirety of $R$. As we will see in the subsequent section, the acyclicity of $\mb{F}^*$ will guarantee this.

\section{Surjectivity of the maps $w^{(i)}$}\label{sec:surjectivity-pf}

In this section we will prove the main technical result of this paper, which underlies the classification in subsequent sections.

\begin{thm}\label{thm:surjectivity}
	Suppose we have a free resolution
	\[
	\mb{F} \colon 0 \to F_3 \to F_2 \to F_1 \to F_0
	\]
	of Dynkin format $\underline{f} = (f_0,f_1,f_2,f_3) = (r_1,r_1+r_2,r_2+r_3,r_3)$ over a $\mb{C}$-algebra $R$, whose dual $0 \to F_0^* \to F_1^* \to F_2^* \to F_3^*$ is also acyclic. Let $\Rgen(\underline{f})$ be the generic ring associated to the format $\underline{f}$, and $w\colon \Rgen(\underline{f}) \to R$ a map specializing $\mb{F}^\mathrm{gen}$ to $\mb{F}$. Then the structure maps
	\begin{gather*}
		w^{(3)}\colon R \otimes L(\omega_{z_{r_3}})^\vee = R \otimes [F_3 \oplus M_0^* \otimes\bigwedge^{f_0+1} F_1 \oplus \cdots] \to R \otimes F_2\\
		w^{(2)} \colon R \otimes L(\omega_{y_{r_2-2}})^\vee = R \otimes [F_1^* \oplus M_0^* \otimes F_3^* \otimes \bigwedge^{f_0} F_1 \oplus \cdots] \to R \otimes F_2^*\\
		w^{(1)} \colon R \otimes L(\omega_{x_{r_1}})^\vee = R \otimes [F_1 \oplus M_0^* \otimes F_3^* \otimes \bigwedge^{f_0+2} F_1 \oplus \cdots] \to R \otimes F_0\\
		w^{(a_3)} \colon R \otimes L(\omega_{z_{1}})^\vee = R \otimes [\bigwedge^{f_3} F_3 \oplus \cdots] \to R \otimes \bigwedge^{f_3} F_2\\
		w^{(a_2)} \colon R \otimes L(\omega_{x_{1}})^\vee = R \otimes [\bigwedge^{r_2} F_1^* \otimes \bigwedge^{f_3} F_3^* \oplus \cdots ]\to R \otimes \bigwedge^{f_2} F_2^* 
	\end{gather*}
	are surjective.
\end{thm}

Note that with the assumptions of the theorem, the Buchsbaum-Eisenbud multiplier $a_1$ yields an isomorphism
\[
a_1\colon M = \bigwedge^{f_3} F_3 \otimes \bigwedge^{f_2} F_2^* \otimes \bigwedge^{f_1} F_1 \xto{\cong} \bigwedge^{f_0} F_0 = M_0.
\]
Using this identification, we have replaced the tensor powers of $M^*$ appearing in the graded decomposition of the critical representations by powers of $M_0^*$ as in Remark~\ref{rem:gl-vs-sl}. If $r_1 = 1$, then this also identifies the map $w^{(1)}$ with $w^{(a_2)}$. In the later sections where we classify perfect ideals, we will primarily make use of the surjectivity of $w^{(1)}$. As it is a map to $R$, its surjectivity is equivalent to being nonzero modulo the maximal ideal $\mf{m} \subset R$.

The prototypical example of Theorem~\ref{thm:surjectivity} mentioned in the introduction is $w^{(2)}$ for $\underline{f}=(1,n,n,1)$. This is a $n \times 2n$ matrix consisting of the differential $d_2$ and an isomorphism $F_1 \cong F_2^*$ induced by a choice of multiplication on $\mb{F}$. The surjectivity of the matrix is evident from the presence of an invertible submatrix.

However for other Dynkin formats, it is hard to generalize this method to find an invertible submatrix. Instead, we will prove surjectivity of the maps $w^{(*)}$ by exhibiting them inside of larger invertible matrices. Conceptually, the idea is that by considering higher structure maps for both $\mb{F}$ and its dual simultaneously, we are able to construct a map $\Spec R \setminus V(H_0(\mb{F})) \to G$ lifting the desired map to $G/P$ from Corollary~\ref{cor:GP-homogeneous-ring-wa2}. Unlike the projective variety $G/P$, the group $G$ itself is affine, thus this map extends to the entirety of $\Spec R$ because $\grade H_0(\mb{F}) = 3 \geq 2$. In the actual proof, we will circumvent dealing with $G$ by working directly with automorphisms of each representation.

\subsection{The setup of the proof}\label{subsec:surjectivity-proof-setup}
For the proof, we fix a map $w'\colon \Rgen(\underline{f}^*) \to R$ specializing $\Fgen(\underline{f}^*)$ to $\mb{F}^*$. As usual, we will identify
\begin{equation}\label{eq:surj-ident-a1}
	\bigwedge^{f_3} F_3 \otimes \bigwedge^{f_2} F_2^* \otimes \bigwedge^{f_1} F_1 \cong \bigwedge^{f_0} F_0
\end{equation}
using $a_1$ following Remark~\ref{rem:gl-vs-sl}. We also make the analogous identification
\begin{equation}\label{eq:surj-ident-a1*}
	\bigwedge^{f_0} F_0^* \otimes \bigwedge^{f_1} F_1 \otimes \bigwedge^{f_2} F_2^* \cong \bigwedge^{f_3} F_3^*
\end{equation}
for $\Rgen(\underline{f}^*)$.

We discuss $w^{(1)}$ as an example, but the situation for the other structure maps is completely analogous. Our strategy is to produce a map $A_1$ making the following diagram commute:
\begin{equation}\label{eq:diamond}
	\begin{tikzcd}
		&L(\omega_{x_{r_1}}) \otimes L(\omega_{x_{r_1}})^\vee \ar[d,dashed,"A_1"]\\
		\Rgen(\underline{f}^*) \supset L(\omega_{x_{r_1}}) \ar[ur,"\Id\otimes \iota"] \otimes F_1 \ar[r,"w'^{(3)}"] & R & \ar[l,"w^{(1)}",swap] F_0^* \otimes L(\omega_{x_{r_1}})^\vee \ar[ul,"\eta \otimes \Id",swap] \subset \Rgen(\underline{f})\\
		& F_0^* \otimes F_1 \ar[u,swap,"d_1"] \ar[ur,"\Id \otimes \iota",swap]\ar[ul,"\eta\otimes \Id"]
	\end{tikzcd}
\end{equation}
where $\iota$ is inclusion of the bottom $z_1$-graded component and $\eta$ is inclusion of the top $x_1$-graded component. Here we are abusing notation somewhat: $w^{(1)}$ is by definition a map $L(\omega_{x_{r_1}})^\vee \otimes R \to F_0 \otimes R$ but we view it in the diagram as a restriction of $w\colon \Rgen \to R$.

Obviously there are many choices of $A_1$ as stated, but the point is to construct it as an isomorphism when viewed as a map $L(\omega_{x_{r_1}})^\vee \otimes R \to L(\omega_{x_{r_1}})^\vee \otimes R$. (Morally, it should be viewed as the action of our desired $R$-point $g \in G$ on $L(\omega_{x_{r_1}})^\vee\otimes R$.) Note that in order to view $A_1$ in this manner, we are using the fact that $T$ is a Dynkin diagram so $L(\omega_{x_{r_1}})^\vee = L(\omega_{x_{r_1}})^*$ since it is finite-dimensional.

\subsection{The split exact case}\label{subsec:splitcase}
Note that if $R = \mb{C}$ and we take $A_1 = \Id$, corresponding in \eqref{eq:diamond} to the evident pairing $L(\omega_{x_{r_1}}) \otimes L(\omega_{x_{r_1}})^\vee \to \mb{C}$, then its restriction to $L(\omega_{x_{r_1}}) \otimes F_1$ recovers $w_\ssc'^{(3)}$ and its restriction to $F_0^* \otimes L(\omega_{x_{r_1}})^\vee$ recovers $w_\ssc^{(1)}$. Restricting down further to $F_0^* \otimes F_1$ recovers the first differential of $\mb{F}^\ssc$, viewed either as $(w_\ssc'^{(3)})_0^*$ if we restrict down via the left half of the diagram, or as $(w_\ssc^{(1)})_0$ if we restrict down the right half.

Hence $A_1 = \Id$ is a natural solution to the lifting problem \eqref{eq:diamond} if $\mb{F} =\mb{F}^\ssc$, $w = w_\ssc$, and $w' = w_\ssc'$. Now suppose $\mb{F}$ is a split exact complex over some ring $R$, with differentials $d_i$. Then one can choose an isomorphism $\mb{F}_\mathrm{ssc} \otimes R \cong \mb{F}$, which amounts to picking $g_1 \in \GL(F_1 \otimes R)$ and $g_2 \in \GL(F_2 \otimes R)$:
\[\begin{tikzcd}
	0 \ar[r] & F_3 \otimes R \ar[r]\ar[d,equals] & F_2 \otimes R \ar[r]\ar[d,"g_2"] & F_1 \otimes R \ar[r]\ar[d,"g_1"] & F_0 \otimes R \ar[d,equals]\\
	0 \ar[r] & F_3 \otimes R \ar[r,"d_3"] & F_2 \otimes R \ar[r,"d_2"] & F_1 \otimes R \ar[r,"d_1"] & F_0 \otimes R
\end{tikzcd}\]
Explicitly $g_1, g_2$ are such that
\[
g_1^{-1} = \begin{blockarray}{cc}
	& \margin{$F_1$} \\
	\begin{block}{c[c]}
		\margin{$F_0$} & d_1\\
		\margin{$C$} & \gamma\\
	\end{block}
\end{blockarray}
\quad\quad
g_2 = \begin{blockarray}{ccc}
	& \margin{$F_3$} & \margin{$C$} \\
	\begin{block}{c[cc]}
		\margin{$F_2$} & d_3 & \beta\\
	\end{block}
\end{blockarray}
\]
with the property that the composite $F_1 \xto{\gamma} C \xto{\beta} F_2$ splits the differential $d_2$.

As discussed above, setting $A_1 = \Id$ in \eqref{eq:diamond} works for $\mb{F}^\ssc \otimes R$. Precompose $w_\ssc \otimes R$ by the action of $g_2 g_1^{-1}$ on $\Rgen(\underline{f}) \otimes R$ to obtain a new map $w_0$, and similarly precompose $w_\ssc' \otimes R$ by the action of $g_2 g_1^{-1}$ on $\Rgen(\underline{f}^*) \otimes R$ to obtain a new map $w_0'$. By Theorem~\ref{thm:parametrize}, $w = w_0 \exp Z^-$ for some $Z^- \in \mbf{L} \cotimes R$ where $\mbf{L} = \hat{\mf{n}}_{z_1}^-$, and similarly $w' = w_0' \exp X^+$ for some $X^+ \in \mbf{L}' \cotimes R$ where $\mbf{L}' = \hat{\mf{n}}_{x_1}^+$. We remark that the completions in defining $\mbf{L}$ and $\mbf{L}'$ are not necessary here because $\mf{g}$ is finite-dimensional.

With the identifications \eqref{eq:surj-ident-a1} and \eqref{eq:surj-ident-a1*}, $\GL(F_2)$ acts only on the left tensor factor in \eqref{eq:diamond} whereas $\GL(F_1)$ only acts on the right tensor factor. So we act on the whole diagram \eqref{eq:diamond} by $g_1 g_2$, thereby replacing $A_1 = \Id$ with
\[
A_1 = \rho_2(g_2) \rho_1(g_1^{-1})
\]
restricting down to our maps $w_0$ and $w_0'$. Here
\begin{align*}
	\rho_1 \colon \GL(F_1) &\to \Aut L(\omega_{x_{r_1}})^\vee\\
	\rho_2 \colon \GL(F_2) &\to \Aut L(\omega_{x_{r_1}})^\vee
\end{align*}
denote the actions of $\GL(F_i)$ on the two tensor factors, viewed as the source and target respectively of $A_1$.

Finally, we act by $\exp X^+$ on the left factor and $\exp Z^-$ on the right factor in the diagram to obtain
\begin{equation}\label{eq:def-A1}
	A_1 = \exp(X^+) \rho_2(g_2) \rho_1(g_1^{-1}) \exp(Z^-)
\end{equation}
restricting down to $w$ and $w'$ by construction. This should be viewed as the action of our desired $R$-point of $G$, as discussed in the motivation at the beginning of our construction. In particular, $A_1$ is obviously invertible by construction, being the composite of invertible maps. Repeating this procedure for the other higher structure maps, we obtain
\begin{itemize}
	\item $A_1$ lifting the pair $w^{(1)}$ and $w'^{(3)}$,
	\item $A_2$ lifting the pair $w^{(2)}$ and $w'^{(2)}$,
	\item $A_3$ lifting the pair $w^{(3)}$ and $w'^{(1)}$,
\end{itemize}
where $A_2$ and $A_3$ are defined using the same formula \eqref{eq:def-A1} but with the actions on $L(\omega_{y_{r_2-2}})^\vee$ and $L(\omega_{z_{r_3}})^\vee$ respectively. The reason for dealing with all $w^{(i)}$ simultaneously is explained by the following lemma.
\begin{lem}\label{lem:matrix-A-comparison}
	Let $\wt{X}^+ \in \mbf{L}' \cotimes R$ and $\wt{Z}^- \in \mbf{L} \cotimes R$. Define $\wt{A}_i$ following \eqref{eq:def-A1}, replacing $X^+$ by $\wt{X}^+$ and $Z^-$ by $\wt{Z}^-$.
	\begin{itemize}
		\item If $\wt{A}_1$ extends $w'^{(3)}$ in the sense of \eqref{eq:diamond}, then $X^+ = \wt{X}^+$.
		\item If $\wt{A}_3$ extends $w^{(3)}$, then $Z^- = \wt{Z}^-$.
	\end{itemize}
\end{lem}
\begin{proof}
	The two statements are completely analogous, so we explain the second one. Observe that the action of $\exp(X^+)$ has no effect on the restriction of $A_3$ to $F_2 \otimes L(\omega_{z_{r_3}})^\vee$, which is $w^{(3)}$. Similarly, the action of $\exp(\wt{X}^+)$ has no effect on the restriction of $\wt{A}_3$ to $F_2 \otimes L(\omega_{z_{r_3}})^\vee$, which is also equal to $w^{(3)}$ by assumption. The statement then follows immediately from Proposition~\ref{prop:HST-determined-by-a3} since $W(a_3) \subset \bigwedge^{r_3} W(d_3)$ in $\Rgen$.
\end{proof}
\subsection{Independence of choice of splitting}\label{subsec:independence-of-splitting}
The construction of the matrices $A_i$ was reliant on a choice of splitting $\mb{F}^\ssc \otimes R \cong \mb{F}$. This was the only step that required a choice; the elements $Z^-, X^+$ were uniquely determined afterwards by comparison of $w_\ssc \otimes R$ to $w$ and $w_\ssc' \otimes R$ to $w'$ using Theorem~\ref{thm:parametrize}.

We now show that \eqref{eq:def-A1} is actually insensitive to our choice of $g_1$ and $g_2$. In the following we will often abuse notation and just write e.g. $F_j$ when we mean $F_j \otimes R$.
\begin{lem}
	Suppose that we pick a different isomorphism $\mb{F}_\mathrm{ssc} \otimes R \cong \mb{F}$, or equivalently, a different splitting $F_1 \xto{\gamma'} C \xto{\beta'} F_2$. Then there exist $\theta \in GL(C)$, $\eta_1 \in \Hom(F_0, C)$, and $\eta_2 \in \Hom(C,F_3)$ such that
	\[
	\gamma' = \theta\gamma + \eta_1 d_1, \quad \beta' = \beta\theta^{-1} + d_3 \eta_2.
	\]
	For the corresponding $g_1', g_2'$, we can write this as
	\[
	g_1'^{-1} = \theta(1+\theta^{-1} \eta_1)g_1^{-1}, \quad g_2' = g_2(1+\eta_2 \theta)\theta^{-1}
	\]
	recalling that $F_1 = F_0 \oplus C$ and $F_2 = F_3 \oplus C$.
\end{lem}
\begin{proof}
	Both $\gamma, \gamma'$ must map $\ker d_1$ isomorphically onto $C$, so there exists an element $\theta \in GL(C)$ such that $\gamma' = \theta \gamma$ restricted to $\ker d_1$. The difference $\gamma' - \theta \gamma$ must then factor through $d_1$. This gives the first expression.
	
	One similarly argues the existence of $\theta' \in GL(C)$ such that $\beta' = \beta\theta'$ modulo $\ker d_2 \subset F_2$. Note that if $s\colon F_1 \to F_2$ is a splitting, then
	\[
	\ker d_1 \hookrightarrow F_1 \xto{s} F_2 \twoheadrightarrow F_2/(\ker d_2)
	\]
	must be inverse to the map induced by $d_2$. In particular, $\beta\gamma$ and $\beta'\gamma'$ must agree as maps $(\ker d_1) \to F_2/(\ker d_2)$, which means $\theta' = \theta^{-1}$. The expression for $\beta'$ thus follows.
\end{proof}
Let $\mf{g}_{m,n}$ denote the part of the Lie algebra $\mf{g}$ in $(x_1,z_1)$-bidegree $(m,n)$. We fix a pair
\[
(i,\omega) \in \{(1,\omega_{x_{r_1}}), (2,\omega_{y_{r_2-2}}), (3,\omega_{z_{r_3}})\}
\]
and let $\rho_1, \rho_2$ denote the actions of $\GL(F_1), \GL(F_2)$ on $L(\omega)^\vee$.

Note that $1+\theta^{-1}\eta_1\in \SL(F_1)$. It can be written as $\exp(\theta^{-1}\eta_1)$, viewing $\theta^{-1}\eta_1 \in F_0^* \otimes C = \mf{g}_{1,0}$. Similarly $1 + \eta_2 \theta = \exp(\eta_2 \theta)$ viewing $\eta_2\theta \in C^* \otimes F_3 = \mf{g}_{0,-1}$.

If we go through the construction of \S\ref{subsec:splitcase} with $g_1',g_2'$, we get
\[
A_i' = \exp (X'^+) \rho_2(g_2') \rho_1(g_1'^{-1}) \exp (Z'^-)
\]
Expanding this using the above observations, we have
\[
A_i' = \exp(X'^+) \rho_2(g_2) \exp (\eta_2 \theta) \rho_2 (\theta^{-1}) \rho_1(\theta) \exp(\theta^{-1} \eta_1) \rho_1(g_1^{-1}) \exp (Z'^-).
\]
Now we use:
\begin{lem}
	The map $\GL(C) \to \GL(F_1) \xto{\rho_1} \Aut L(\omega)^\vee$ agrees with $\GL(C) \to \GL(F_2) \xto{\rho_2} \Aut L(\omega)^\vee$.
\end{lem}
\begin{proof}
	The statement is certainly true for $\SL(C)$ because both actions can be seen through $\mf{sl}(C)\subset \mf{g}_{0,0}$. So it is sufficient to check this statement for scalars, and we omit this.
\end{proof}
Hence $\rho_2(\theta^{-1})$ and $\rho_1(\theta)$ cancel, and we are left with
\[
A_i' = \exp(X'^+) \rho_2(g_2) \exp (\eta_2 \theta) \exp(\theta^{-1} \eta_1) \rho_1(g_1^{-1}) \exp (Z'^-).
\]
Elements of $\mf{g}_{1,0}$ and $\mf{g}_{0,-1}$ commute because $\mf{g}_{1,-1} = 0$, so we can interchange the middle two terms. Note that
\begin{align*}
	\theta^{-1}\eta_1 &\in \mf{g}_{1,0} = F_0^* \otimes C\\
	&\subset \mf{g}_{1,*} = F_0^* \otimes \bigwedge^{f_3+1}F_2 \otimes \bigwedge^{f_3} F_3^*
\end{align*}
Applying $g_2$ to $\theta^{-1}\eta_1$ gives an element $X_1\in \mf{g}_{1,*}$ such that
\[
\exp (X_1) \rho_2(g_2) = \rho_2(g_2) \exp(\theta^{-1}\eta_1).
\]
Similarly, by applying $g_1$ to $\eta_2 \theta$, we get $Z_1 \in \mf{g}_{*,-1}$ such that
\[
\rho_1(g_1^{-1}) \exp (Z_1) = \exp(\eta_2\theta) \rho_1(g_1^{-1}),
\]
allowing us to write
\[
A_i' = \exp(X'^+) \exp (X_1) \rho_2(g_2) \rho_1(g_1^{-1}) \exp(Z_1) \exp(Z'^-).
\]
Baker-Campbell-Hausdorff yields elements $\widetilde{X}^+ \in \mbf{L}' \cotimes R$ and $\wt{Z}^- \in \mbf{L} \cotimes R$ such that
\[
\exp (\wt{X}^+) = \exp(X'^+)\exp(X_1), \quad \exp (\wt{Z}^-) = \exp (Z_1)\exp (Z'^-),
\]
and so
\[
A_i' = \exp(\wt{X}^+) \rho_2(g_2) \rho_1(g_1) \exp(\wt{Z}^-).
\]
However, compare this to
\[
A_i = \exp(X^+) \rho_2(g_2) \rho_1(g_1) \exp(Z^-).
\]
The hypotheses of Lemma~\ref{lem:matrix-A-comparison} are met by construction, and we deduce that $X^+ = \wt{X}^+$ and $Z^- = \wt{Z}^-$. In particular $A_i = A_i'$ as desired.

Having established that the matrices $A_i$ are independent of the choice of splitting, Theorem~\ref{thm:surjectivity} readily follows.
\begin{proof}[Proof of Theorem~\ref{thm:surjectivity}]
	We fix $w$ and $w'$ as in \S\ref{subsec:surjectivity-proof-setup}. Let $h_1,h_2 \in I_{f_0}(d_1)$ be a regular sequence, which exists because $\grade I_{f_0}(d_1)  = 3 \geq 2$. We perform the construction of \S\ref{subsec:splitcase} for the split exact complex $\mb{F} \otimes R_{h_1}$, obtaining matrices $A_i$ invertible over $R_{h_1}$. We repeat the construction for $\mb{F} \otimes R_{h_2}$, obtaining matrices $A_i'$ invertible over $R_{h_2}$. The results of this subsection then show that the constructions agree on the common localization $R_{h_1 h_2}$. Hence the matrices $A_i = A_i'$ have entries in $R_{h_1} \cap R_{h_2} = R$. The same applies to their inverses $A_i^{-1} = A_i'^{-1}$, therefore the matrices $A_i$ are invertible over $R$. Since $w^{(i)}$ consists of rows from $A_i$, we obtain that $w^{(i)}$ is surjective.
\end{proof}

\subsection{First applications}\label{subsec:first-apps}
The major applications of Theorem~\ref{thm:surjectivity} will be explored in later sections. We conclude this section with a easy but surprising consequence in the graded setting: the theorem gives a restriction on graded Betti numbers. Let $k$ be a field of characteristic zero, $R = k[x_0,\ldots,x_m]$, and $M$ a graded Cohen-Macaulay $R$-module. Suppose the graded minimal free resolution of $M$ has the form
\[
0 \to \bigoplus_{j=1}^{b_3} R(-s_{3j}) \to \bigoplus_{j=1}^{b_2} R(-s_{2j}) \to \bigoplus_{j=1}^{b_1} R(-s_{1j}) \to \bigoplus_{j=1}^{b_0} R(-s_{0j}) \to M \to 0.
\]
\begin{cor}\label{cor:grading-restriction}
	In the setup above, if $(b_0,b_1,b_2,b_3)$ is Dynkin, then the graded module
	\[
	R \otimes L(\omega_{x_{1}})^\vee = \bigwedge^{f_2} F_2  \otimes [\bigwedge^{r_2} F_1^* \otimes \bigwedge^{f_3} F_3^* \oplus \cdots ]
	\]
	must have a generator in degree zero.
\end{cor}
\begin{proof}
	For a graded resolution, it is possible to pick $w\colon \Rgen \to \mb{C} \otimes R$ such that the higher structure maps are all homogeneous of degree zero; see Proposition~\ref{prop:equivariant-p} and Example~\ref{ex:non-perfect}. Since Theorem~\ref{thm:surjectivity} implies that $w^{(a_2)}$ is nonzero mod $\mf{m} = (x_1,\ldots,x_m)$, some entry of the matrix must have degree exactly zero.
\end{proof}
One case of particular interest is $F_0 = R$, so $M = R/I$. As $a_1$ is an isomorphism, the above equivalently says that
\[
R \otimes L(\omega_{x_{1}})^\vee = F_1 \oplus F_3^* \otimes \bigwedge^{3} F_1 \oplus \cdots,
\]
the domain of $w^{(1)}$, has a generator in degree zero. Each graded component is a subrepresentation of $F_1 \otimes \mf{g}_1^{\otimes j}$ for some $j \geq 0$, where $\mf{g}_1 = F_3^* \otimes \bigwedge^2 F_1$.
\begin{example}
	Suppose $R/I$ is Cohen-Macaulay and $(1,b_1,b_2,b_3)$ is Dynkin. Since the module $F_1$ is generated in positive degrees, it follows that $F_3^* \otimes \bigwedge^2 F_1$ must have a generator in negative degree in order for $R \otimes L(\omega_{x_{1}})^\vee$ to have a generator in degree zero. Thus $2\min(s_{1j}) < \max(s_{3j})$.
	
	From the perspective of linkage, this example can also be obtained as a corollary of Theorem~\ref{thm:licciconj}, which we will establish in a forthcoming paper. The inequality $2\min(s_{1j}) < \max(s_{3j})$ then follows from \cite[Corollary 5.13]{Huneke-Ulrich87}.
\end{example}
\begin{example}
	If $R/I$ is Cohen-Macaulay and $(1,b_1,b_2,b_3)$ is Dynkin, then at least one $s_{1j}$ is even or at least one $s_{3j}$ is odd. If this were not the case, $F_1$ would be generated in odd degree and $F_3^* \otimes \bigwedge^2 F_1$ would be generated in even degree. Consequently $R \otimes L(\omega_{x_1})^\vee$ would have all generators in odd degree, thus none in degree zero.
\end{example}

\section{Classification of perfect ideals with Dynkin Betti numbers}\label{sec:classify}
We will now apply Theorem~\ref{thm:surjectivity} to classify perfect ideals $I$ of grade 3 in a local ring $(R,\mf{m},k)$ of equicharacteristic zero (c.f. Assumption~\ref{ass:base-field}), with the assumption that $R/I$ has Betti numbers in the Dynkin range. If $R$ is regular, this is equivalent to $R/I$ being Cohen-Macaulay, with one of the following conditions on the type $t$ and deviation $d$:
\begin{itemize}
	\item $t = 1$,
	\item $d = 1$,
	\item $t = 2$ and $d \leq 4$, or
	\item $t \leq 4$ and $d = 2$.
\end{itemize}
The first two cases are already well-understood. The novel results come from the last two---these correspond to Dynkin types $E_n$.

We will make use of the important connection between $\Rgen$ and the homogeneous space $G/P_{x_1}^+$ laid out in Corollary~\ref{cor:GP-homogeneous-ring-wa2}. Throughout this whole section, we fix a Dynkin format $(1,f_1,f_2,f_3) = (1,3+d,2+d+t,t)$. Let $T_{2,d+1,t+1}$ be the corresponding diagram
\[\begin{tikzcd}
	x_1 \ar[r,dash] & u \ar[r,dash] \ar[d,dash] & y_1 \ar[r,dash] & \cdots \ar[r,dash] & y_d\\
	& z_1 \ar[d,dash]\\
	& \vdots \ar[d,dash]\\
	& z_t
\end{tikzcd}\]
Let $V = L(\omega_{x_1})$, and let $v \in V$ be a highest weight vector. By abuse of notation we also write $v \in G/P_{x_1}^+ \subset \mb{P}(V)$ for its span, i.e. the Borel-fixed point.

Let $I \subset R$ be a perfect ideal of grade 3 such that $R/I$ is resolved by $\mb{F}$ of the format $(1,f_1,f_2,f_3)$. We do not require that $\mb{F}$ be minimal. Let $w\colon \Rgen \to \mb{F}$ be a map specializing the generic resolution to $\mb{F}$. From Theorem~\ref{thm:surjectivity} we know that $w^{(1)}$ is surjective, equivalently nonzero mod $\mf{m}$, and hence by Corollary~\ref{cor:GP-homogeneous-ring-wa2} it yields a map $\Spec R \to G/P_{x_1}^+$. In particular, by looking at $w^{(1)} \otimes k$, where $k = R/\mf{m}$ is the residue field, we get a $k$-point of $G/P_{x_1}^+$. The next result says that the $P_{z_1}^-$-orbit (c.f. \S\ref{subsec:cells-P-orbits}) containing $w^{(1)} \otimes k$ is well-defined and preserved under local specialization.

\begin{prop}\label{prop:classification-by-lowest-unit}
	Suppose that $\mb{F}$ has Dynkin format and resolves $R/I$ for a perfect ideal $I \subset R$. Let $w\colon \Rgen \to R$ specialize the generic resolution to $\mb{F}$. The $P_{z_1}^-$-orbit containing the $k$-point determined by $w^{(1)} \otimes k$ depends only on the ideal $I \subset R$ and not on the choice of resolution $\mb{F}$ or map $w\colon \Rgen \to R$ specializing the generic resolution to $\mb{F}$.
	
	If $I' \subset R'$ is a perfect ideal in regular local ring $(R',\mf{m}',k')$ of equicharacteristic zero and $R \xto{\varphi} R'$ is a local homomorphism such that $\varphi(I)R' = I'$, then the orbits determined by $I$ and $I'$ are the same.
\end{prop}
\begin{proof}
	For the first point, let $\mb{F}'$ be a different resolution of format $(1,f_1,f_2,f_3)$ for $R/I$, and $w'\colon \Rgen \to R$ a map specializing the generic resolution to $\mb{F}'$. We view $w,w'$ as maps $R \otimes \Rgen \to R$. By precomposing $w$ with the action of an appropriate element $g\in \prod \GL(F_i)$ on $R \otimes \Rgen$, we can arrange so that $wg \colon R \otimes \Rgen \to R$ specializes the generic resolution to $\mb{F}'$. Since both $w'$ and $wg$ have this property, it follows from Theorem~\ref{thm:parametrize} that there is an element $X \in \mb{L} \otimes R$ such that $wg\exp X = w'$. Both the actions of $g$ and $\exp X$ on $V$ preserve the $P_{z_1}^-$ orbit that contains $w^{(1)} \otimes k$, and the claim follows.
	
	For the other point of the proposition, simply note that $\mb{F} \otimes R'$ is a resolution of $R'/I'$, and
	\[
	(w^{(1)} \otimes_R R') \otimes_{R'} k' = (w^{(1)} \otimes_R k) \otimes_k k'
	\]
	because we assumed $\varphi$ to be a local homomorphism, so it induces an inclusion of residue fields $k \hookrightarrow k'$.
\end{proof}

Let $W$ denote the Weyl group of $G$ and, for $j \in T_{2,d+1,t+1}$, $W_{P_{j}} \subset W$ the subgroup generated by all simple reflections $\{s_i\}_{i\neq j}$. The Schubert cells of $G/P_{x_1}^+$ are indexed by the torus-fixed points. In the Pl\"ucker embedding $G/P_{x_1}^+ \hookrightarrow \mb{P}(V)$, these are exactly the extremal weight lines, which are in correspondence with $W/W_{P_{x_1}}$.

The $P_{z_1}^-$-orbits in $G/P_{x_1}^+$, each of which is a union of Schubert cells as defined in \S\ref{subsec:cells-P-orbits}, are indexed by the double cosets $W_{P_{z_1}} \backslash W / W_{P_{x_1}}$. If $\sigma \in W$ is the minimal length representative of such a double coset, then $\sigma v \in V$ is a highest weight vector for an extremal representation of $\mf{sl}(F_3) \times \mf{sl}(F_1)$ inside of $V$. In this manner, the $P_{z_1}^-$-orbits correspond to these extremal representations, and we obtain the following algebraic translation of the above proposition, in the language of higher structure maps.
\begin{prop}
	Suppose that $\mb{F}$ has Dynkin format and resolves $R/I$ for a perfect ideal $I \subset R$. Let $w\colon \Rgen \to R$ specialize the generic resolution to $\mb{F}$. In the $z_1$-graded decomposition
	\[
	R \otimes L(\omega_{x_{1}})^\vee = F_1 \oplus F_3^* \otimes \bigwedge^{3} F_1 \oplus \cdots
	\]
	there is a lowest irreducible $\mf{gl}(F_3)\times \mf{gl}(F_1)$-representation to which the restriction of $w^{(1)}$ is nonzero mod $\mf{m}$. This representation depends only on $I\subset R$ and is necessarily extremal.
\end{prop}

To summarize, so far we have demonstrated how a perfect ideal determines an element of $W_{P_{z_1}} \backslash W / W_{P_{x_1}}$ which can be used to classify the ideal. Next we will show that every double coset is realizable in this manner, and exhibit a generic perfect ideal for each double coset.
\begin{thm}\label{thm:generic-examples}
	As in \S\ref{bg:sch-schubert-var}, let $w = s_{z_1}s_u s_{x_1} \in W$ and let $X^w \subset G/P_{x_1}^+$ be the corresponding codimension 3 Schubert variety. Let $\sigma \in W$ be a representative of a double coset in $W_{P_{z_1}}\backslash W / W_{P_{x_1}}$ and let $S_\sigma \coloneqq \mc{O}_{G/P_{x_1}^+,\sigma v}$ be the local ring of $G/P_{x_1}^+$ at the torus-fixed point $\sigma v$; it is isomorphic to a polynomial ring localized at its ideal of variables $\mf{m}_\sigma$. Let $I_\sigma$ be the ideal of $X^w$ at that point. It is the unit ideal if $[\sigma] = [e]$ where $e \in W$ is the identity. Otherwise it is a perfect ideal of grade 3 in $S_\sigma$. It has the following properties:
	\begin{enumerate}
		\item If $w\colon \Rgen \to S_\sigma$ specializes the generic resolution to a resolution of $S_\sigma / I_\sigma$, the point of $G/P_{x_1}^+$ determined by $w^{(1)} \otimes S_\sigma/\mf{m}_\sigma$ is in $P_{z_1}^-\sigma v$, the $P_{z_1}^-$-orbit corresponding to the double coset $[\sigma]$.
		\item If $(R,\mf{m},k)$ is a local ring of equicharacteristic zero, $w\colon \Rgen \to R$ specializes the generic resolution to a resolution of $R/I$ for a perfect ideal $I$, and $w^{(1)}\otimes k$ is in the same $P_{z_1}^-$-orbit as $\sigma v$, then there exists a local homomorphism $\varphi \colon S_\sigma \to R$ such that $I = \varphi(I_\sigma)R$.
	\end{enumerate}
\end{thm}
\begin{proof}
	Note that $X^w$ is the union of all $P_{z_1}^-$-orbits aside from the big open orbit $P_{z_1}^- v$. So $I_\sigma$ is the unit ideal when $[\sigma]=[e]$. If $[\sigma] \neq [e]$ then $\sigma v \in X^w$ and it is well-known that Schubert varieties are Cohen-Macaulay \cite{Ramanathan85}, so $I_\sigma \subset S_\sigma$ is a perfect ideal.
	
	To prove (1), in view of Proposition~\ref{prop:classification-by-lowest-unit}, it suffices to construct one such $w$ and verify the statement. We will produce this $w$ using the action of $\mf{g}$ on $\Rgen$. We begin with the map $w_\mathrm{ssc} \colon \Rgen \to \mb{C}$ from Theorem~\ref{thm:ssc}, and observe that its restriction $w_\mathrm{ssc}^{(1)} \colon V^\vee \to \mb{C}$ is dual to a highest weight vector $v \in V$. Let $\mf{n}_{x_1}^-$ be the negative part of $\mf{g}$ in the $x_1$-grading. The open cell $C^e = B^- v \subset G/P_{x_1}^+$ is the orbit of $v$ under the exponential action of $\mf{n}_{x_1}^-$, thus it can be identified with $\Spec \Sym (\mf{n}_{x_1}^-)^*$.
	
	Next we produce a map $\Rgen \to \Sym (\mf{n}_{x_1}^-)^*$ by precomposing
	\[
	\Rgen \xto{w_\mathrm{ssc}} \mb{C} \to \Sym(\mf{n}_{x_1}^-)^*
	\]
	with the action of $\exp X$ on $\Rgen \otimes \Sym(\mf{n}_{x_1}^-)^*$, where
	\[
	X \in \mf{n}_{x_1}^- \otimes (\mf{n}_{x_1}^-)^* \subset \mf{n}_{x_1}^- \otimes \Sym(\mf{n}_{x_1}^-)^*
	\]
	is the ``generic element'' of $\mf{n}_{x_1}^-$, i.e. $X$ is adjoint to the identity on $\mf{n}_{x_1}^-$. 
	
	Finally we precompose this map by the action of $\sigma^{-1} \in W$ (or more accurately, a representative thereof as in Remark~\ref{rem:Weyl-action-lift}) on $\Rgen$. Let $w\colon \Rgen \to \Sym (\mf{n}_{x_1}^-)^*$ be the result. By construction, $w^{(1)} \colon V^* \to \Sym (\mf{n}_{x_1}^-)^*$ is none other than a parametrization of the open patch $\sigma \exp(\mf{n}_{x_1}^-)\cdot [v]$, where $\sigma v$ is the origin, corresponding to the ideal of variables in $\Sym (\mf{n}_{x_1}^-)^*$. In particular, $S_\sigma$ is the localization of this polynomial ring at its ideal of variables. The ideal $I_\sigma$ is generated by the Pl\"ucker coordinates coming from the bottom $z_1$-graded component $w^{(1)}_0$ of $w^{(1)}$.
	
	The map $w$ specializes the generic resolution to a complex over $S_\sigma$ with differentials $w^{(1)}_0$, $w^{(2)}_0$, and $w^{(3)}_0$. This complex is none other than the resolution of $S_\sigma / I_\sigma$ given in Theorem~\ref{thm:intermediate-res}.
	
	Point (2) follows readily from the discussion before Proposition~\ref{prop:classification-by-lowest-unit}. If $w \colon \Rgen \to R$ is such a map, and $p$ is the $k$-point of $G/P_{x_1}^+$ determined by $w^{(1)} \otimes k$, then the ideal of $X^w \subset G/P_{x_1}^+$ at the point $p$ specializes to the ideal $I$. The important point is that the actions of $\GL(F_i)$ and $\exp \mb{L}$ on $G/P$ preserve the Schubert variety $X^w$. Consequently, since $p$ and $\sigma v$ are in the same $P_{z_1}^-$-orbit, the local defining equations of $X^w$ at these two points are equivalent up to a change of coordinates. Thus $I_\sigma$ specializes to $I$ as well.
\end{proof}
We will not explicitly describe the ideals $I_\sigma$ in this paper; they can get very complicated and there are simply too many of them for $E_7$ and $E_8$. The paper \cite{NW-examples} outlines how to produce free resolutions of $S_{\sigma}/I_\sigma$ for $\sigma = w_0$ the longest element, and that construction is easily adapted to other $\sigma$.

However, we will at least tie back our results to the discussion in the introduction \S\ref{sec:intro} and describe the situation for Dynkin types $D_n$ and $E_6$. Before doing so, we note that even without explicitly understanding the ideals $I_\sigma$, we obtain the following classification result as a corollary of the above.
\begin{thm}\label{thm:double-coset-classification}
	Fix a Dynkin format $(1,f_1,f_2,f_3)$ and the corresponding setup as in the beginning of this section. Every element of $W_{P_{z_1}} \backslash W / W_{P_{x_1}} - [e]$ describes a non-empty family of perfect ideals of grade 3 with Betti numbers $(1,b_1,b_2,b_3)$, where $b_i \leq f_i$ for all $i$. These families are disjoint, and this is the finest possible classification that is preserved under local specialization.
\end{thm}
\begin{proof}
	If $I \subset R$ is perfect of grade 3, the condition that $b_i \leq f_i$ is equivalent to saying that $R/I$ admits a (not necessarily minimal) resolution of format $(1,f_1,f_2,f_3)$, and so its classification comes from Proposition~\ref{prop:classification-by-lowest-unit}. The non-emptiness of each family comes from Theorem~\ref{thm:generic-examples} point (1). Disjointness comes from Proposition~\ref{prop:classification-by-lowest-unit}, as does the fact that this classification is preserved under local specialization. The existence of a generic example for each family, Theorem~\ref{thm:generic-examples} point (2), shows that it is the finest classification with this property.
\end{proof}

Now we revisit the types $D_n$ and $E_6$, starting with the two families of $D_n$ formats.
\begin{example}\label{ex:D_n-1}
	Consider the format $(1,n,n,1)$ where $n \geq 3$. The corresponding diagram is $T_{2,n-2,2} = D_n$, and the representation $V$ is a half-spinor representation. The $z_1$-graded decomposition of $V^*$ into $\mf{gl}(F_3) \times \mf{gl}(F_1)$-representations is
	\[
	V^* = (F_1) \oplus (F_3^* \otimes \bigwedge^3 F_1) \oplus (S_2 F_3^* \otimes \bigwedge^5 F_1) \oplus \cdots \oplus (S_{\lfloor{\frac{n-1}{2}}\rfloor} F_3^* \otimes \bigwedge^{2\lfloor{\frac{n-1}{2}}\rfloor+1} F_1)
	\]
	Every representation appearing is extremal; they correspond to the elements of $W_{P_{z_1}} \backslash W / W_{P_{x_1}}$. Aside from the lowest representation $F_1$, which corresponds to $[e] \in W_{P_{z_1}} \backslash W / W_{P_{x_1}}$, there are $\lfloor\frac{n-1}{2}\rfloor$ extremal representations; each one of these is a possible location for the lowest appearance of a unit in the structure map $w^{(1)}$.
	
	If $I \subset R$ is a perfect ideal of grade 3 such that $R/I$ has Betti numbers $(1,b_1,b_2,b_3)$ with $b_i \leq f_i$, then necessarily $b_3 = 1$ and $b_1=b_2 \leq n$. Gorenstein ideals of grade 3 are minimally generated by an odd number of elements \cite{Watanabe73}. Moreover, for each odd $b_1$ with $3 \leq b_1 \leq n$, the generic example of such an ideal is given by \cite{Buchsbaum-Eisenbud77}, confirming Theorem~\ref{thm:double-coset-classification} for this format.
\end{example}
\begin{example}\label{ex:D_n-2}
	Consider the format $(1,4,n,n-3)$ where $n \geq 4$. The corresponding diagram is $T_{2,2,n-2} = D_n$, and the representation $V$ is again a half-spinor representation. However, the node $z_1$ is different from the preceding example, and the decomposition of $V^*$ into $\mf{gl}(F_3) \times \mf{gl}(F_1)$-representations is
	\[
	V^* = (F_1) \oplus (F_3^* \otimes \bigwedge^3 F_1) \oplus (\bigwedge^2 F_3^* \otimes S_{2,1^3} F_1) \oplus \cdots \oplus (\bigwedge^{n-3} F_3^* \otimes S_{(a+1)^b,a^{(4-b)}} F_1).
	\]
	Here $a = \lfloor \frac{n-3}{2}\rfloor$, and $b = 2 + (-1)^n$. Every representation appearing is extremal. Again excluding the lowest representation $F_1$, we see $n-3$ possible representations for the lowest appearance of a unit in $w^{(1)}$. These correspond to the Betti numbers $(1,3,3,1)$ and $(1,4,b_2,b_2-3)$ where $5 \leq b_2 \leq n$. There is a generic perfect ideal with each of these Betti numbers; see \cite{Buchsbaum-Eisenbud77} and \cite{Brown87}.
\end{example}

We refer the reader to \cite{GW20} for a detailed discussion of higher structure maps for the preceding two formats, including methods for computing $w^{(i)}_j$ explicitly via lifting. Next we turn our attention to the $E_6$ format.

\begin{example}\label{ex:E6}
	Consider the format $(1,5,6,2)$. The corresponding diagram is $T_{2,3,3} = E_6$, and the representation $V$ is the adjoint. The $z_1$-graded decomposition of $V^*$ into $\mf{gl}(F_3) \times \mf{gl}(F_1)$-representations is
	\[
	V^* = (F_1) \oplus (F_3^* \otimes \bigwedge^3 F_1) \oplus \begin{bmatrix}
		(S_2 F_3^* \otimes \bigwedge^5 F_1)\\
		\oplus (\bigwedge^2 F_3^* \otimes S_{2,1^3} F_1)\\
		\oplus (\bigwedge^2 F_3^* \otimes \bigwedge^5 F_1)
	\end{bmatrix} \oplus (S_{2,1} F_3^* \otimes S_{2^2,1^3} F_1) \oplus (S_{2,2} F_3^* \otimes S_{2^4,1} F_1) 
	\]
	All representations, except for the $\bigwedge^2 F_3^* \otimes \bigwedge^5 F_1$ appearing in the middle, are extremal.
	
	If $R/I$ has Betti numbers $(1,b_1,b_2,b_3)$ with $b_i \leq f_i$, then the Betti numbers can be one of $(1,3,3,1)$, $(1,5,5,1)$, $(1,4,5,2)$, or $(1,5,6,2)$. The first three have already been discussed in the preceding two examples; in particular there is a generic example for each one. This leaves \emph{two} elements of $W_{P_{z_1}} \backslash W / W_{P_{x_1}}$ for the Betti numbers $(1,5,6,2)$, namely those corresponding to the last two representations $S_{2,1} F_3^* \otimes S_{2^2,1^3} F_1$ and $S_{2,2} F_3^* \otimes S_{2^4,1} F_1$.
	
	If one explicitly computes the ideals $I_\sigma$ for these two cases, the two can be distinguished by whether the multiplication on $\Tor_1(S_\sigma/I_\sigma,S_\sigma/\mf{m}_\sigma)$ is nonzero. For the former example, the multiplication is nonzero, and the ideal is described in \cite[Theorem 4.4]{Brown87}. This multiplication is zero for the latter, and that $I_\sigma$ is the ideal $J(t)$ in \cite{CLKW20}. Its relationship to $E_6$ is discussed at length in that paper as well as in \cite[\S3.2]{NW-examples}. Since the property of this Tor algebra multiplication being (non)zero is preserved under local specialization as discussed in \S\ref{sec:intro}, it can be used to distinguish between the corresponding two families of $(1,5,6,2)$ perfect ideals.
\end{example}
The double cosets $W_{P_{z_1}} \backslash W / W_{P_{x_1}}$ can be computed algorithmically, and we thank Witold Kra\'skiewicz for providing us with Python code that does so. By looking at the cardinality of this set, a simple counting argument can be used to deduce the number of families of perfect ideals with Betti numbers corresponding to types $E_7$ and $E_8$. We have already witnessed the first point of the following theorem in Example~\ref{ex:E6}.
\begin{thm}\label{thm:counts}
	In the sense of Theorem~\ref{thm:double-coset-classification}, there are:
	\begin{itemize}
		\item 2 families of perfect ideals with Betti numbers $(1,5,6,2)$,
		\item 7 families of perfect ideals with Betti numbers $(1,6,7,2)$,
		\item 11 families of perfect ideals with Betti numbers $(1,5,7,3)$,
		\item 49 families of perfect ideals with Betti numbers $(1,7,8,2)$,
		\item 90 families of perfect ideals with Betti numbers $(1,5,8,4)$.
	\end{itemize}
\end{thm}
\begin{proof}
	In Table~\ref{table:num-d-t}, we have summarized the cardinality of $W_{P_{z_1}} \backslash W / W_{P_{x_1}}$ for various Dynkin formats $(1,3+d,2+d+t,t)$; call this quantity $\#(d,t)$.
	
	\begin{table}[htbp]
		\caption{Cardinality $\#(d,t)$ of $W_{P_{z_1}} \backslash W / W_{P_{x_1}}$ for $T_{2,d+1,t+1}$ associated to Dynkin formats $(1,3+d,2+d+t,t)$}
		\centering
		\begin{tabular}{c|cccccc}
			\toprule
			$\#(d,t)$ & $t=1$ & $t=2$ & $t=3$ & $t=4$ & $t=5$ & $\cdots$ \\
			\hline
			$d=0$ & 2 & 2 & 2 & 2 & 2 & $\cdots$\\
			$d=1$ & 2 & 3 & 4 & 5 & 6 & $\cdots$\\
			$d=2$ & 3 & 6 & 18 & 109 & -- \\
			$d=3$ & 3 & 13 & -- & -- & -- \\
			$d=4$ & 4 & 63 & -- & -- & -- \\
			$d=5$ & 4 & -- & -- & -- & --\\
			$\vdots$ & \vdots \\
			\bottomrule
		\end{tabular}
		\label{table:num-d-t}
	\end{table}
	
	Since $\#(d,t)-1$ records the number of families of perfect ideals with Betti numbers \emph{at most} $(1,3+d,2+d+t,t)$, we have that the number of families with Betti numbers \emph{exactly} the given format is
	\[
	\#(d,t) - \#(d-1,t) - \#(d,t-1) + \#(d-1,t-1).
	\]
	The counts in the theorem follow easily.
\end{proof}

We close this section by mentioning a non-local form of Theorem~\ref{thm:generic-examples}, to address the ``genericity conjecture'' that has appeared in previous work on the subject, e.g. \cite[Questions 4.9]{WeymanICERM}.

In \cite{NW-examples}, the focus was on $\sigma = w_0 \in W$, the longest element. The reason for focusing on $w_0$ was a matter of perspective, rather than any technical limitation. Rather than looking at the local defining equations of $X^w$ at any particular point, both that paper and its precursor \cite{SW21} examined the ideal of $Y^w \coloneqq X^w \cap w_0 C^e$ inside of $w_0 C^e$, where $C^e = B^- v$ is the big open Schubert cell and $w_0 C^e$ is its opposite. This yields an ideal $J_{w_0} \subset S \coloneqq \Sym(\mf{n}_{x_1}^-)^*$. The ``origin'' $w_0 v$ of this open cell $w_0 C^e$ is a point in the lowest-dimensional $P_{z_1}^-$-orbit of $G/P_{x_1}^+$. So by localizing at the ideal of variables, we recover the ideal $I_{w_0}\subset S_{w_0}$.

Loosely speaking, the genericity conjecture says that a general choice of higher structure maps for a perfect ideal with Dynkin Betti numbers will have a unit in the top coordinate of $w^{(1)}$. This follows easily from Theorem~\ref{thm:surjectivity}, but it carries less information than Theorem~\ref{thm:generic-examples}. We present and prove it mainly for the sake of closing the loop in this circle of ideas.

\begin{thm}\label{thm:non-local-gen}
	Fix a Dynkin format $(1,f_1,f_2,f_3)$. If $\mb{F}$ is a resolution of $R/I$ with the given format, where $I\subset R$ is a perfect ideal of grade 3 in a local ring of equicharacteristic zero, then there is a map $\varphi\colon S \to R$ such that $\varphi(J_{w_0})R = I$.
\end{thm}
Note that the polynomial ring $\Sym(\mf{n}_{x_1}^-)^*$ is \emph{not} localized, as we are considering the whole affine patch $w_0 C^e$. Since $J_{w_0}$ and $I$ are both perfect ideals of grade 3, the resolution for $S/J_{w_0}$ constructed in \cite{NW-examples} specializes to one for $R/I$ via $\varphi$. This is the form in which the statement appears in \cite[Conjecture 2.6]{NW-examples}. 
\begin{proof}
	Pick a map $w\colon \Rgen \to R$ specializing the generic resolution to $\mb{F}$, and let $w^{(1)}$ denote its restriction to $W(d_1) = V^\vee$ as usual.
	
	Let $\lambda$ be a highest weight vector of $V^\vee$, i.e. dual to $w_0 \cdot v \in V$. Note that $\ker(w^{(1)} \otimes k)$ is a hyperplane in $V^\vee$ since $w^{(1)} \otimes k \neq 0$. The linear span of the $B^-$-orbit of $\lambda$ is the entirety of $V^\vee$, thus a general element $g\in B^-$ has the property that $g\cdot \lambda \notin \ker(w^{(1)} \otimes k)$.
	
	Precomposing $w^{(1)}$ by such a $g$ does not change the image of $w^{(1)}_0$ because the bottom $z_1$-graded component $F_1\subset V^\vee$ is preserved under the action of $B^-$. Hence $w^{(1)}g$ determines a map $\Spec R \to G/P_{x_1}$ landing in $w_0 C^e$, and the corresponding map $S \to R$ specializes $J_{w_0}$ to $I$ as desired.
\end{proof}

\section{Next steps}\label{sec:beyond-ADE}

The theory of $\Rgen$ has been developed for arbitrary resolution formats $(f_0,f_1,f_2,f_3)$ of length 3, not just those in the Dynkin range. Actually, the construction of $A_i$ in the proof of Theorem~\ref{thm:surjectivity} works fine without the Dynkin hypothesis, the caveat being that they must be thought of as maps
\[
A_i \colon L(\omega_i) \otimes L(\omega_i)^\vee \to R.
\]
Note that the graded dual $L(\omega_i)^\vee$ is strictly smaller than the ordinary dual $L(\omega_i)^*$ if the representation is infinite-dimensional. In particular, such a map $A_i$ cannot be interpreted as an endomorphism of $L(\omega_i)^\vee$, and thus the crucial final step of the proof, in which we argue that the $A_i$ are invertible, does not make sense.

Phrasing it in this manner, one might be led to think that this is a technical shortcoming of the proof. However, the theorem statement itself is not even true without the Dynkin hypothesis, as the following example shows.
\begin{example}\label{ex:nonlicci-perfect}
	Let $I = (x,y,z)^2 \subset R= \mb{C}[x,y,z]$. If we take $R$ with the standard grading, then $R/I$ admits a graded minimal free resolution
	\[
	\mb{F}\colon 0 \to R(-4)^3 \to R(-3)^8 \to R(-2)^6 \to R.
	\]
	The format $(1,6,8,3)$ is not Dynkin, as it is associated to the affine type $\wt{E}_7$. The quotient $R/I$ is obviously Cohen-Macaulay, being zero-dimensional, so $\mb{F}^*$ is acyclic. Thus the other assumption of Theorem~\ref{thm:surjectivity} is still satisfied.
	
	However, we know that there exists $\Rgen \to R$ resulting in $w^{(1)}$ being homogeneous of degree zero. Analogously to Example~\ref{ex:non-perfect},
	\[
	R \otimes L(\omega_{x_1})^\vee = F_1 \oplus F_3^* \otimes \bigwedge^3 F_1 \oplus \cdots
	\]
	has all generators in degree 2, since $F_1$ is generated in degree 2 and $F_3^* \otimes \bigwedge^2 F_1$ is generated in degree 0. Thus the entries of $w^{(1)}$ are all quadrics by degree considerations and $w^{(1)}$ is not surjective. Indeed, its image is equal to $I$.
\end{example}

In a future paper, we will elucidate a connection between higher structure maps coming from $\Rgen$ and the theory of linkage. From that perspective, the surjectivity of $w^{(1)}$ is equivalent to the ideal $I$ being in the linkage class of a complete intersection (licci), and Theorem~\ref{thm:surjectivity} implies the following:
\begin{thm}\label{thm:licciconj}
	Let $I$ be a grade 3 perfect ideal in a local Noetherian ring $R$ of equicharacteristic zero. Let $d$ denote the deviation of $I$ and $t$ the minimal number of generators of $\operatorname{Ext}^3(R/I,R)$. If
	\begin{itemize}
		\item $d \leq 4$ and $t \leq 2$, or
		\item $d \leq 2$ and $t \leq 4$,
	\end{itemize}
	then $I$ is in the linkage class of a complete intersection.
\end{thm}

On the other hand, the ideal $(x,y,z)^2$ from the preceding example is \emph{not} licci. Moreover, in \cite{CVW19} it is shown that for all $(1,f_1,f_2,f_3)$ outside the Dynkin range, there exists a perfect ideal with those Betti numbers that is not licci. Thus the Dynkin condition in Theorem~\ref{thm:surjectivity} is essential.

Beyond the Dynkin range, it remains unclear how to use representation theory to characterize non-licci perfect ideals. A concrete starting point would be to see whether one can produce some well-known examples of non-licci perfect ideals with Betti numbers $\underline{f} = (1,6,8,3)$ directly from the representation theory of $E_7^{(1)}$, which is the affine Kac-Moody Lie algebra involved in the construction of $\Rgen(\underline{f})$. Two particularly simple examples of such perfect ideals are
\begin{itemize}
	\item the ideal of $2\times 2$ minors of a generic $2\times 4$ matrix (the ideal of $\mb{P}^1 \times \mb{P}^3 \subset \mb{P}^7$ in the Segre embedding), and
	\item the ideal of $2\times 2$ minors of a generic $3\times 3$ symmetric matrix (the ideal of $\mb{P}^2 \subset \mb{P}^5$ in the Veronese embedding).
\end{itemize}

Next we discuss a different avenue for future work. The behavior of perfect ideals of codimension $c$ often parallels the behavior of Gorenstein ideals of codimension $c+1$. Indeed, there are various methods of producing the latter given an example of the former. From this perspective, after studying perfect ideals of codimension 3, a natural next step is to examine Gorenstein ideals of codimension 4.

For a such an ideal $I$ generated by $n$ elements, and $\mb{F}$ a self-dual resolution of $R/I$, there is a ring $A(n)_\infty$ and a map $A(n)_\infty \to R$ which can be viewed as a collection of ``higher structure maps'' for $\mb{F}$. This construction is described in \cite{WeymanGor4}. The Kac-Moody Lie algebra associated to the diagram $T_{3,n,2}$ acts on $A(n)_\infty$. In particular, when $6 \leq n \leq 8$, this Lie algebra is $E_n$. In analogy with Theorem~\ref{thm:surjectivity}, we conjecture that the higher structure maps, suitably interpreted, are surjective for $n \leq 8$. Assuming this, one can develop a similar classification as we have done in \S\ref{sec:classify}.

However, for the proof of Theorem~\ref{thm:surjectivity}, it was necessary to have \emph{two} rings $\Rgen(\underline{f})$ and $\Rgen(\underline{f}')$ to assemble the matrices $A_i$. In a way, it would appear that the ring $A(n)_\infty$ only provides half of the picture needed to carry out this same program. For example, Kustin's ``higher order products'' are not visible in $A(n)_\infty$, and we expect to find them coming from another ring. We do not yet have a systematic construction of this second set of structure maps in general, but for $n = 6$, they can be explicitly computed.

\appendix
\section{Some lemmas pertaining to $\Rgen$}\label{sec:Rgen-pfs}
\begin{thm}\label{thm:parametrize}
	Let $\mb{F}$ be a resolution of length three over $R$ and let $\Rgen$ be the generic ring for the associated format. Fix a $\mb{C}$-algebra homomorphism $w\colon \Rgen \to R$ specializing $\mb{F}^\mathrm{gen}$ to $\mb{F}$. Then $w$ determines a bijection
	\[
	\mbf{L}\cotimes R \coloneqq \prod_{i > 0} (\mb{L}_i \otimes R) \simeq \{\text{$\mb{C}$-algebra homomorphisms $w'\colon \Rgen \to R$ specializing $\mb{F}^\mathrm{gen}$ to $\mb{F}$}\}.
	\]
	Note that a $\mb{C}$-algebra homomorphism $\Rgen \to R$ can be viewed as an $R$-algebra homomorphism $\Rgen \otimes R\to R$. The correspondence above identifies $X \in \mbf{L} \cotimes R$ with the map $w\exp X$ obtained by precomposing $w$ with the action of $\exp X$ on $\Rgen\otimes R$. 
\end{thm}
\begin{proof}
	By Lemma~\ref{lem:GFR-p-determines-w}, the homomorphism $w\colon \Rgen \otimes R \to R$ is completely determined by the choice of the structure maps $p_i$. For $X \in \mbf{L} \cotimes R$, let us write $X = \sum_{i > 0} u_i$ where $u_i \in \mb{L}_i \otimes R$, and let $X_n = \sum_{i=1}^n u_i$ denote the partial sums.
	
	Precomposing $w$ by $\exp X$ or $\exp X_n$ has the same effect on the structure maps $p_k$ for $k \leq n$. Acting by $\exp X$ on $p_1$, we get
	\[
	p_1 + (\bigwedge^{r_3} d_3)u_1^*.
	\]
	Here $u_1^*$ means the dual of $R \xto{u_1} \mb{L}_1 \otimes R$. All possible choices of the structure map $p_1$ are obtained by lifting a particular map $q_1$ in the diagram \eqref{eq:p-lifting}, so it follows that choices of $u_1 \in \mb{L}_1 \otimes R$ correspond to choices for the structure map $p_1$.
	
	Once $X_{n-1}$ has been computed, $u_n \in \mb{L}_n \otimes R$ can be similarly determined by comparing $p_n$ with $p_n'$. Acting by $\exp X$ on $p_n$ gives
	\[
	(p_n + p_{n-1}[u_1,-]^* + \cdots) + (\bigwedge^{r_3} d_3)u_n^*.
	\]
	The first part consists of terms involving $u_k$ for $k<n$, which have already been determined. Once again, \eqref{eq:p-lifting} shows that there is a unique choice of $u_n \in \mb{L}_n \otimes R$ that makes the whole expression equal to $p_n'$.
	
	Proceeding inductively in this fashion, we construct $X \in \mb{L} \otimes R$ with the desired property, and the uniqueness at each step is evident as well.
\end{proof}

We next establish some straightforward representation theory lemmas, which will be used to highlight the importance of a particular subspace in $W(a_3) \subset \Rgen$.
\begin{lem}
	Let $b \in L(\omega_{z_1})^\vee$ be a lowest weight vector. The subspace
	\[
	V \coloneqq \{Xb : X \in \mf{g}\} \subseteq L(\omega_{z_1})^\vee
	\]
	is a representation of $\mf{n}_{z_1}^-$, and thus a $\mbf{L}$-representation.
\end{lem}
\begin{proof}
	Let $Y \in \mf{n}_{z_1}^-$. Then
	\[
	YX b = [Y,X]b - XYb = [Y,X]b
	\]
	since $Yb= 0$. The representation $L(\omega_{z_1})^\vee$ is a lowest weight representation, so for any $Y' \in \hat{\mf{n}}_{z_1}^-$, there exists a truncation $Y \in \mf{n}_{z_1}^-$ of $Y$ such that $Yb = Y'b$, so $\mbf{L}$ also acts on $V$.
\end{proof}
\begin{lem}
	The map
	\[
	\mb{C} \oplus \mb{L}^\vee \to V
	\]
	sending $1 \in \mb{C}$ to $b$ and $X \in \mb{L}^\vee = \mf{n}_{z_1}^+$ to $Xb$ is an isomorphism of vector spaces.
\end{lem}
\begin{proof}
	It is easy to see that the subspace
	\[
	\mf{p} \coloneqq \{ X \in \mf{g} : Xb \in \mb{C}b\} \subset \mf{g}
	\]
	is a subalgebra. Moreover it contains the maximal parabolic $\mf{p}_{z_1}^- = \bigoplus_{\alpha \leq_{z_1} 0} \mf{g}_\alpha$. Since $L(\omega_{z_1})^\vee$ is not the trivial representation, we have $\mf{p}_{z_1}^- \subseteq \mf{p} \subsetneq \mf{g}$. Therefore $\mf{p} = \mf{p}_{z_1}^-$, so $\mb{L}^\vee \cap \mf{p} = 0$ and the map is injective.
	
	Given any $X \in \mf{g}$, we may express $X$ as $X^+ + X^-$ where $X^+ \in \mf{n}_{z_1}^+$ and $X^- \in \mf{p}_{z_1}^-$. Then $Xb = X^+ b + X^- b$ where $X^- b \in \mb{C}b$, showing surjectivity.
\end{proof}
\begin{remark}\label{rem:p-in-a3}
	We expect that the restriction of $w\colon \Rgen \to R$ to
	\[
	\bigwedge^{r_2} F_2^* \otimes \bigwedge^{r_3} F_3 \otimes \mb{L}_m^* \subset \bigwedge^{r_2} F_2^* \otimes \bigwedge^{r_3} F_3 \otimes [\mb{C} \oplus \mb{L}^\vee] \subseteq W(a_3)
	\]
	should exactly recover the structure map $p_m$. While this should not be too difficult to prove, it would be somewhat technical and unnecessary for our purposes, so we leave it as a guess. Some of the results below should (in principle) be consequences of this statement, but of course we do not assume this statement in their proofs.
\end{remark}

\begin{lem}\label{lem:V-exp-props}
	Let $h \in R$ be a nonzerodivisor, let $\pi \colon [\mb{C} \oplus \mb{L}^\vee] \otimes R \twoheadrightarrow \mb{C} \otimes R$ be projection onto the first factor, and let $\gamma\colon \mb{L}^\vee \otimes R \to \mb{C} \otimes R$ be any map. Then
	\begin{itemize}
		\item there is a unique $X \in \mbf{L} \cotimes R_h$ such that $(h\pi+\gamma) = h\pi \exp X$, and
		\item if $S$ is a ring containing $R$ and $X' \in \mbf{L} \cotimes S$ satisfies $(h\pi+\gamma) = h\pi \exp X'$, then $X'$ must belong to $\mbf{L} \cotimes R_h$ and thus equal $X$.
	\end{itemize}
\end{lem}
\begin{proof}
	For $X \in \mbf{L} \cotimes S$, we have $\pi X = 0$ only when $X = 0$. Thus, the precomposition action of $\exp (\mbf{L}\cotimes S)$ on $\pi$ has trivial stabilizer, showing uniqueness. One can solve for $X$ explicitly in a manner similar to the proof of Theorem~\ref{thm:parametrize}, showing it must be an element of $\mbf{L}\cotimes R_h$. One should informally think of $X$ as $\log (\pi + \gamma/h)$; we omit the details.
\end{proof}

The next result says that we do not lose any information by only considering the structure maps $w^{(i)}$, since they uniquely determine $w$. In fact, we only need the differentials and part of $w^{(a_3)}$, which we recall can be computed from $w^{(3)}$ (c.f. Remark~\ref{rem:a3-in-d3}).

\begin{prop}\label{prop:HST-determined-by-a3}
	Let $w$ and $w'$ be two maps $\Rgen \to R$ specializing $\Fgen$ to the same resolution $\mb{F}$ over some ring $R$. Viewing them as maps $\Rgen \otimes R \to R$, write $\bar{w},\bar{w}'$ for their restrictions to
	\[
	\bigwedge^{r_2} F_2^* \otimes \bigwedge^{r_3} F_3 \otimes [\mb{C} \oplus \mb{L}^\vee] \otimes R \subseteq W(a_3) \otimes R.
	\]
	Then there is a unique element $X \in \mbf{L} \cotimes R$ such that $\bar{w}' = \bar{w}\exp X$.
	
	In particular, if $\bar{w} = \bar{w}'$, then $X = 0$ and $w = w'$.
\end{prop}
The existence of such an element $X$ is already known by Theorem~\ref{thm:parametrize}; the substance of this statement is that $X$ is completely determined by comparing $\bar{w}$ and $\bar{w}'$. If Remark~\ref{rem:p-in-a3} were true, this would be immediate given Lemma~\ref{lem:GFR-p-determines-w}.
\begin{proof}
	Since $\mb{F}$ is acyclic, $\grade I_{r_3}(d_3) = 3 \geq 1$ so there is some $e \in \bigwedge^{r_3}F_3\otimes \bigwedge^{r_2}F_2^*$ such that $h = a_3(e) \in R$ is a nonzerodivisor. Let $\bar{w}_e$ and $\bar{w}'_e$ denote the restrictions of $\bar{w}$ and $\bar{w}'$ to $\mb{C}e \otimes [\mb{C} \oplus \mb{L}^\vee]$, viewed as maps
	\[
	[\mb{C} \oplus \mb{L}^\vee] \otimes R \to (\mb{C}e)^* \otimes R \cong R
	\]
	From Theorem~\ref{thm:parametrize}, we know there exists $X \in \mbf{L} \cotimes R$ with the property that $\bar{w}_e = \bar{w}_e' \exp X$. The uniqueness follows from Lemma~\ref{lem:V-exp-props}.
\end{proof}

It will often be convenient to manipulate higher structure maps $w^{(i)}$ over a larger ring containing $R$, for instance a localization in which $\mb{F}$ becomes split exact. The next result ensures that, even if we work in a larger ring, it is easy to tell when $w$ factors through $R$.

\begin{prop}\label{prop:localization-lemma}
	Let $R \subset S$ be two rings, $\mb{F}$ a resolution over $R$ with the property that $\mb{F}\otimes S$ is also acyclic, and $w' \colon \Rgen \to S$ a map specializing $\Fgen$ to $\mb{F} \otimes S$.
	
	Suppose the restriction of $w'$ to
	\[
	\bigwedge^{r_2} F_2^* \otimes \bigwedge^{r_3} F_3 \otimes [\mb{C} \oplus \mb{L}^\vee] \subseteq W(a_3)
	\]
	is $R$-valued. Then $w'$ factors through $R$.
\end{prop}
\begin{proof}
	This proof is similar to the preceding one. This time we use that $\grade I_{r_3}(d_3) = 3 \geq 2$ so there are $e_1,e_2 \in \bigwedge^{r_3}F_3\otimes \bigwedge^{r_2}F_2^*$ such that $h_1 = a_3(e_1)$ and $h_2 = a_3(e_2)$ form a regular sequence.
	
	Since $\mb{F}$ is acyclic, we may pick a $w\colon \Rgen \to R$ specializing $\Fgen$ to $\mb{F}$. By Theorem~\ref{thm:parametrize}, there exists $X \in \mbf{L} \cotimes S$ such that $w' = (w\otimes S)\exp X$.
	
	Let $\bar{w}_e$ and $\bar{w}_e'$ denote the restrictions of $w \otimes S$ and $w'$ to $(\mb{C}e_1 \oplus \mb{C}e_2) \otimes [\mb{C} \oplus \mb{L}^\vee] \subset W(a_3)$, viewed as maps
	\[
	[\mb{C} \oplus \mb{L}^\vee] \otimes S \to (\mb{C}e_1 \oplus \mb{C}e_2)^* \otimes S \cong S^2.
	\]
	We have $\bar{w}_e' = \bar{w}_e \exp X$. Applying Lemma~\ref{lem:V-exp-props} to the first row of $\bar{w}_e$ and $\bar{w}_e'$, we find that $X \in \mbf{L} \cotimes R_{h_1}$. Applying it to the second row, we find that $X \in \mbf{L} \cotimes R_{h_2}$. Since $h_1,h_2$ is a regular sequence, we have $R_{h_1} \cap R_{h_2} = R$ and $X \in \mbf{L} \cotimes R$, from which it follows that $w'$ factors through $R$ if viewed as a map from $\Rgen$.
\end{proof}

\printbibliography
	
\end{document}